\documentclass[a4paper,11pt,reqno]{amsart}

\usepackage{anysize,color}
\usepackage[utf8]{inputenc}
\usepackage[dvipsnames]{xcolor}
\usepackage{float}
\usepackage{amsfonts,amssymb,amsmath}
\usepackage{amsthm}
\usepackage{url}
\usepackage{paralist}
\usepackage{tikz}
    \usetikzlibrary{decorations.markings,arrows,shapes.callouts,calc}
\usepackage{rotating}
\usepackage[mathscr]{euscript}
\usepackage{charter}
\usepackage{verbatim}
\usepackage{tikz-cd}
\usepackage{mathtools}
\usepackage{tabularx}
\IfFileExists{soul.sty}{\usepackage{soul}}{}
\usepackage[arrow,curve,matrix,tips,2cell]{xy}
  \SelectTips{eu}{10} \UseTips
  \UseAllTwocells
\IfFileExists{bbold.sty}{\usepackage{bbold}}{}
\usepackage{booktabs}

\newtheorem*{theorem*}{Theorem}
\newtheorem{theorem}{Theorem}[section]
\newtheorem{proposition}[theorem]{Proposition}
\newtheorem{lemma}[theorem]{Lemma}

\newtheorem{corollary}[theorem]{Corollary}
\newtheorem{conjecture}[theorem]{Conjecture}
\makeatletter\let\c@equation\c@theorem\makeatother

\theoremstyle{definition}
\newtheorem{definition}[theorem]{Definition}
\theoremstyle{remark}

\newcommand\R{\mathbb R}
\newcommand\C{\mathbb C}
\newcommand\Z{\mathbb Z}

\newcommand\Q{\mathbb Q}
\newcommand\FF{\mathbb F}
\newcommand\sym{\mathfrak S}
\newcommand\K{\mathscr{K}}
\newcommand\LL{\mathscr{L}}
\newcommand\pt{\mathrm{pt}}
\newcommand\Tv{\mathrm{Tv}}
\newcommand\E{\mathbb{E}}
\newcommand\CC{\mathscr{C}}
\newcommand\Eperp{(\E^G)^{\perp}}

\DeclareMathOperator{\vertex}{\mathrm{vert}}
\DeclareMathOperator{\conv}{\mathrm{conv}}
\DeclareMathOperator{\aff}{\mathrm{aff}}

\DeclareMathOperator{\pr}{\mathrm{pr}}
\DeclareMathOperator{\supp}{\mathrm{supp}}

\begin{document}
\dedicatory{Dedicated to the memory of  Vera and Milivoj Blagojevi\'c}
\title[The number of Tverberg partitions]{The number of Tverberg partitions  beyond prime powers}

\author[Blagojevi\'c]{Matija  Blagojevi\'{c}}
\thanks{The research by Matija  Blagojevi\'{c} leading to these results has
received funding from the Studienstiftung des deutschen Volkes.}
\address{Zuse Institut, 14195 Berlin, Germany}
\email{matib03@zedat.fu-berlin.de}

\author[Blagojevi\'c]{Pavle V. M. Blagojevi\'{c}}
\thanks{The research by Pavle V. M. Blagojevi\'{c} leading to these results has
received funding from the Serbian Ministry of Science, Technological development and Innovations. It was also supported by G\"unter Ziegler's group at the Freie Universit\"at Berlin ``Arbeitsgruppe Diskrete Geometrie und Topologische Kombinatorik''.}
\address{Inst. Math., Freie Universit\"at Berlin, Arnimallee 2, 14195 Berlin, Germany\hfill\break
	\mbox{\hspace{4mm}}Mat. Institut SANU, Knez Mihailova 36, 11001 Beograd, Serbia}
\email{blagojevic@math.fu-berlin.de}
%-------------------------------------------------------------------%
%-------------------------------------------------------------------%
\begin{abstract}
Sierksma's conjecture asserts that every affine map $a\colon\Delta_{(d+1)(r-1)}\longrightarrow\R^d$ has at least $((r-1)!)^d$ unordered Tverberg partitions into $r$ parts.
Until now, lower bounds growing with $d$ were known only when $r$ is a prime power.
As one of our main results we prove that for every integer $r\geq3$, every decomposition $r=kq+s$ into $q=p^\alpha$ a prime power, $k\geq1$ and $s\geq0$, and every affine map $a$, the number $\Tv(a)$ of unordered Tverberg partitions satisfies
\[
	\Tv(a)\ \geq\ \frac{q}{r!}\,\Big(\frac{q}{\alpha+1}\Big)^{\big\lceil\frac{(r-1-2s)(d+1)}{2}\big\rceil}.
\]

\end{abstract}

\maketitle

%----------------------------------------------------------------------------------------------
%----------------------------------------------------------------------------------------------
\bigskip
\section{Introduction and statement of the main results}
\label{sec : introduction}
\bigskip
%----------------------------------------------------------------------------------------------
%----------------------------------------------------------------------------------------------

The classical 1966 theorem of Tverberg \cite{Tverberg1966}, which is a generalisation of the even more classical 1921 theorem of Radon \cite{Radon1921}, has fostered some of the most intriguing research advancements on the crossroads of convex geometry and equivariant algebraic topology.
Since Tverberg's original work, the result has seen many new proofs and variations, as well as generalisations to the topological, colourful, tolerant, and even parametrised settings.
Many Tverberg type questions remain open, with limited progress and even without clear new approaches towards a complete solution, for example Sierksma's Dutch cheese problem \cite{Sierksma1979,Reay1982} and the B\'ar\'any--Larman conjecture \cite{Barany1992}.
Recent highlights in the area include the final solution of the topological Tverberg conjecture \cite{Blagojevic2019, Frick2015, Mabillard2014} and the partial solution of the B\'ar\'any--Larman conjecture \cite{Blagojevic2009}.
The history of the developments that emerged from Tverberg's theorem has been discussed in several surveys and popular articles with varying focuses.
A comprehensive and complete presentation of the results related to Tverberg's theorem would breach the scope of this paper, so for more insight into the history of the problem and the solution methods, we refer the reader, for example, to the publications \cite{Barany2016, BaranySoberon2018, BlagojevicZiegler2017, Matousek2008, Ziegler2011}.

\medskip
Tverberg's original result, phrased in the language of affine maps, states the following.

\medskip
\begin{theorem}
	\label{th : Tverberg's original}
	Let $r\geq 2$, $d\geq 1$ and $N:=(d+1)(r-1)$ be integers, $\Delta_N$ an $N$-dimensional simplex and let $a\colon \Delta_N\longrightarrow\R^d$ be an affine map.
	Then there are $r$ pairwise disjoint faces $F_1,\dots, F_r$ of the simplex $\Delta_N$ such that $a(F_1)\cap\cdots\cap a(F_r)\neq\emptyset$.
\end{theorem}

\medskip
For an affine map $a\colon \Delta_N\longrightarrow\R^d$ a collection of pairwise disjoint faces $(F_1,\dots, F_r)$ of the simplex $\Delta_N$ with the property that $a(F_1)\cap\cdots\cap a(F_r)\neq\emptyset$ is called an (ordered) {\em Tverberg partition} of the map $a$.
By permuting the collection $(F_1,\dots, F_r)$, we obtain in total $r!$ Tverberg partitions with the same faces. Therefore we call the set $\{F_1,\dots, F_r\}$ an {\em unordered Tverberg partition}.
For the number of unordered Tverberg partitions of the affine map $a$ we write $\Tv(a)$.
We do not require that the faces $F_1,\dots,F_r$ cover the vertex set of $\Delta_N$; for affine maps in general position this holds automatically, as explained below.

\medskip
An affine map $a\colon \Delta_N\longrightarrow\R^d$ is determined by the $N+1$ points $a(v)$, for $v\in\vertex(\Delta_N)$, and $a(F)=\conv (a(\vertex (F)))$ for every face $F$ of $\Delta_N$.
Theorem~\ref{th : Tverberg's original} is actually a reformulation of Tverberg's original statement in the language of affine maps, which claimed that $N+1$ points $a(v)$ can be split into $r$ groups, indexed by pairwise disjoint sets of vertices, whose convex hulls have a common point.
Throughout the paper we use the language of affine maps.

\medskip
In this paper we focus on Sierksma's Dutch cheese problem \cite{Sierksma1979}, which conjectures a lower bound for the number of unordered Tverberg partitions of an arbitrary affine map.
More precisely, in his notes from 1979 Sierksma considered the following affine map $a\colon\Delta_N\longrightarrow\R^d$, the so-called $(d+1)$-clouds configuration: split the $N+1=(d+1)(r-1)+1$ vertices of $\Delta_N$ into $d+1$ groups of $r-1$ vertices and one further vertex, map the vertices of the $i$-th group to points close to the $i$-th vertex of a $d$-dimensional simplex in $\R^d$, and map the remaining vertex to a point in the interior of this simplex; for an illustration see Figure~\ref{fig : clouds}.
This map has exactly $((r-1)!)^d$ unordered Tverberg partitions, and Sierksma conjectured that it is extremal; consult~\cite{Sierksma1979}, \cite{Reay1982} and~\cite[Sec.\,4.1]{BaranySoberon2018}.

\begin{figure}[ht]
\centering
\begin{tikzpicture}[scale=1.0,
	pt/.style={circle,fill=black,inner sep=1.4pt},
	ctr/.style={circle,draw,line width=1.1pt,inner sep=2.4pt},
	lab/.style={font=\scriptsize}]
	\coordinate (v1) at (0,0); \coordinate (v2) at (5,0); \coordinate (v3) at (2.5,4.2);
	\coordinate (c) at (2.5,1.4);
	\draw[black!30] (v1) -- (v2) -- (v3) -- cycle;
	\foreach \v/\a/\b in {v1/0.22/0.32, v2/-0.3/0.28, v3/0.05/-0.4}{
		\fill[black!10] (\v) circle (0.42);}
	\coordinate (a1) at (0.22,0.32); \coordinate (a2) at (0.08,-0.2);
	\coordinate (b1) at (4.72,0.28); \coordinate (b2) at (5.12,-0.22);
	\coordinate (c1) at (2.55,3.8); \coordinate (c2) at (2.3,4.45);
	\draw[line width=1.1pt] (a1) -- (b2) -- (c2) -- cycle;
	\draw[line width=1.1pt,black!55,dash pattern=on 2.5pt off 1.8pt] (a2) -- (b1) -- (c1) -- cycle;
	\foreach \q in {a1,a2,b1,b2,c1,c2}{\node[pt] at (\q) {};}
	\node[pt] at (c) {}; \node[ctr] at (c) {};
	\node[lab,anchor=east] at (-0.5,0.05) {$C_1$};
	\node[lab,anchor=west] at (5.5,0.05) {$C_2$};
	\node[lab,anchor=south] at (2.5,4.75) {$C_3$};
	\node[lab,anchor=north] at (2.5,1.15) {$c$};
\end{tikzpicture}
\caption{\small The $(d+1)$-clouds configuration for $d=2$ and $r=3$: three clouds $C_1,C_2,C_3$ of $r-1=2$ points near the vertices of a triangle, and the barycentre $c$, in all $N+1=7$ points.
Its unordered Tverberg partitions consist of the singleton $\{c\}$ and two triangles, drawn as a solid and a dashed triangle; both contain $c$.
There are $(2!)^2=4$ such partitions, one for each way of matching the points of $C_2$ and of $C_3$ with the two points of $C_1$.}
\label{fig : clouds}
\end{figure}
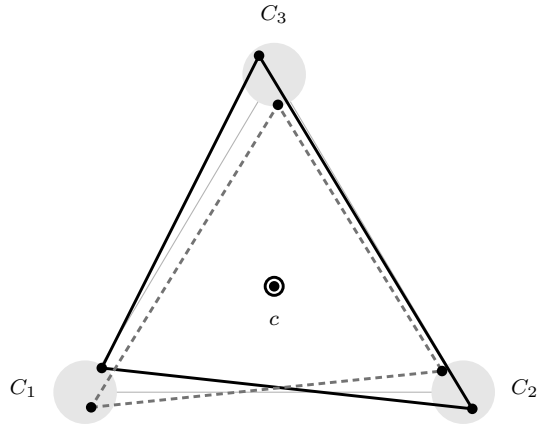

\medskip
\begin{conjecture}[Sierksma's conjecture]
	\label{conj : Sierksma}
	Let $r\geq 2$, $d\geq 1$ and $N:=(d+1)(r-1)$ be integers, and let $a\colon \Delta_N\longrightarrow\R^d$ be an affine map.
	Then 
	\[
		\Tv(a)\ \geq\ ((r-1)!)^d .
	\]
\end{conjecture}

\medskip
For integers $r\geq2$ and $d\geq1$ let
\[
	\Tv_r(d)\ :=\ \min\big\{\Tv(a)\ :\ a\colon\Delta_{(d+1)(r-1)}\longrightarrow\R^d\ \text{an affine map}\big\}
\]
denote the smallest number of unordered Tverberg partitions of an affine map in dimension $d$.
This minimum exists since $\Tv(a)$ is a positive integer for any affine map $a$ by Theorem~\ref{th : Tverberg's original}.
The $(d+1)$-clouds configuration shows that $\Tv_r(d)\leq((r-1)!)^d$, so Sierksma's conjecture actually states that $\Tv_r(d)=((r-1)!)^d$.

\medskip
If true, the conjectured bound is attained by many affine maps, as demonstrated by White~\cite{White2017}.
Lower bounds for $\Tv(a)$ are known only in a few situations.
The first one is due to Vu\v{c}i\'c and \v{Z}ivaljevi\'c~\cite{VucicZivaljevic1993} for primes, and it was extended to prime powers by Hell~\cite{Hell2007}.
Both results hold also for continuous maps $\Delta_N\longrightarrow\R^d$.
Here and throughout the paper a {\em prime power} is an integer $p^\alpha$ with $p$ a prime and $\alpha\geq1$.

\medskip
\begin{theorem}
	\label{th : VZ Hell}
	Let $r\geq 2$, $d\geq 1$ and $N:=(d+1)(r-1)$ be integers, and let $a\colon \Delta_N\longrightarrow\R^d$ be an affine map.
	If $r=p^\alpha$ is a prime power, then
	\[
		\Tv(a)\ \geq\ \frac{1}{(r-1)!}\Big(\frac{r}{\alpha+1}\Big)^{\lceil \frac{N}2\rceil}.
	\]
\end{theorem}

\medskip
Note that the bound of Vu\v{c}i\'c--\v{Z}ivaljevi\'c and Hell, for a fixed prime power $r$, grows exponentially in $d$ with rate $(\frac{r}{\alpha+1})^{(r-1)/2}$, while the conjectured bound grows with rate $(r-1)!$.
Since $\log\big((\frac{r}{\alpha+1})^{(r-1)/2}\big)=(\tfrac12+o(1))\log\,(r-1)!$ as $r\to\infty$, for large $r$ the bound of Vu\v{c}i\'c--\v{Z}ivaljevi\'c and Hell is, on the exponential scale, roughly the square root of the conjectured bound.
Indeed, $\log\,(r-1)!=(r-1)\log(r-1)-(r-1)+O(\log r)=(1+o(1))\,r\log r$ by Stirling's formula, while $\log\big((\frac{r}{\alpha+1})^{(r-1)/2}\big)=\frac{r-1}{2}\big(\log r-\log(\alpha+1)\big)=(1+o(1))\,\frac r2\log r$, because $r=p^\alpha\geq2^\alpha$ gives $\alpha\leq\log_2r$ and hence $\log(\alpha+1)=O(\log\log r)$.

\medskip
Some of the results we recall below assume that the affine map is in general position.
Since an affine map $a\colon\Delta_N\longrightarrow\R^d$ is determined by the images of the $N+1$ vertices of $\Delta_N$, the space of all affine maps $\Delta_N\longrightarrow\R^d$ is isomorphic to $(\R^d)^{N+1}$.
We say that an affine map $a\colon\Delta_N\longrightarrow\R^d$ is in {\em general position} if
\begin{compactitem}[\quad --]
	\item every set of at most $d+1$ vertices of $\Delta_N$ is mapped to affinely independent points, and
	\item for all pairwise disjoint nonempty subsets $S_1,\dots,S_r$ of the vertices of $\Delta_N$ with
	\[
	|S_1|\leq d+1, \dots, |S_r|\leq d+1
	\qquad\text{and}\qquad
	\sum_{j=1}^r(d+1-|S_j|)>d
	\]
	we have
	\[ \aff (a(S_1))\cap\dots\cap\aff (a(S_r))=\emptyset.\]
\end{compactitem}
A dimension count shows that the maps which are not in general position are contained in a proper algebraic subset of $(\R^d)^{N+1}$: the first condition fails on the zero set of finitely many determinants, and for fixed sets $S_1,\dots,S_r$ as in the second condition the affine subspaces $\aff(a(S_1)),\dots,\aff(a(S_r))$, whose codimensions add up to more than $d$, have a common point only if a system of more than $d$ linear equations in $d$ unknowns, whose coefficients are polynomials in $a$, is solvable; this forces the augmented matrix of the system, which has $d+1$ columns, to have rank at most $d$, a closed algebraic condition which fails for a generic $a$.
Hence the maps in general position are dense.
Other papers use other, sometimes stronger, conditions; for example, in~\cite{Hell2008constraints} all coordinates of the points $a(v)$ are required to be independent over $\Q$.

\medskip
For affine maps in general position every Tverberg partition $(F_1,\dots,F_r)$ uses all vertices of $\Delta_N$, and each face has at most $d+1$ vertices.
Indeed, let $m_j$ be the number of vertices of $F_j$.
Applying the second condition to the vertex sets of the faces, where a face with more than $d+1$ vertices is replaced by $d+1$ of its vertices, gives $\sum_j\max\{0,d+1-m_j\}\leq d$, while $\sum_j(d+1-m_j)\geq r(d+1)-(N+1)=d$.
Hence $\sum_jm_j=N+1$ and $m_j\leq d+1$ for all $j$.
So in general position a Tverberg partition is a partition of the vertex set of $\Delta_N$ into $r$ faces.
Moreover, the requirement of being in general position does not affect lower bounds on $\Tv(a)$.
Indeed, suppose $\Tv(b)\geq m$ for all affine maps $b$ in a dense subset of $(\R^d)^{N+1}$, let $a$ be arbitrary, and choose maps $b_t\to a$ in that subset.
Since $\Delta_N$ has only finitely many collections of faces, the same $m$ unordered collections are Tverberg partitions of $b_t$ for infinitely many $t$, and they remain Tverberg partitions of $a$, because the condition $b_t(F_1)\cap\cdots\cap b_t(F_r)\neq\emptyset$ passes to the limit $a$, and so $\Tv(a)\geq m$.
%All results of this paper are stated and proved for arbitrary affine maps.

\medskip
Further results on Sierksma's conjecture include the following.
\begin{compactitem}[\quad --]
	\item Sierksma's conjecture holds for $r=3$ and $d=2$, that is, $\Tv(a)\geq4$ for every affine map $a\colon\Delta_6\longrightarrow\R^2$. This was first established by Hell~\cite[Thm.\,6]{Hell2008constraints}, see also~\cite[Rem.\,3.19]{Hell2006}, and recently an elementary proof was given by Sober\'on~\cite{Soberon2026}. As far as we know, this is the only non-trivial case of the conjecture which has been confirmed.
	\item For $r$ a prime power, $r>2$, $d\geq2$ and affine maps in general position, Hell~\cite[Thm.\,5]{Hell2008constraints} proved the lower bound $\min\{(r-1)!,\,c_{d,r}(r-d)!\}$ with an integer constant $c_{d,r}\geq2$.
	\item For an arbitrary $r$ and affine maps in general position, Hell~\cite{Hell2008Birch} proved, in his work on the number of Birch partitions, that $\Tv(a)\geq(r-d)!$. This is the first and, as far as we know, the only lower bound valid for all $r$; see~\cite[Sec.\,4.1]{BaranySoberon2018}. In particular, when $r$ is not a prime power and $d\geq r-1$, no lower bound for $\Tv(a)$ is known which is better than the trivial bound $\Tv(a)\geq1$.
\end{compactitem}

\medskip
Consequently, beyond prime powers, very little is known about the number of unordered Tverberg partitions.
Moreover, the counterexamples to the topological Tverberg conjecture~\cite{Blagojevic2019, Frick2015, Mabillard2014, MabillardWagner-III, AvvakumovKarasevSkopenkov2023} show that for $r$ not a prime power and $d$ large enough there are continuous maps $\Delta_N\longrightarrow\R^d$ without a single Tverberg partition.
Hence any lower bound for $\Tv(a)$ beyond prime powers has to use the affine structure of the map $a$ in an essential way.

\medskip
The central result of this paper is the following lower bound for the number of unordered Tverberg partitions, which holds for every number of parts $r$.

\medskip
\begin{theorem}
	\label{th : main}
	Let $r\geq3$, $d\geq1$ be integers, $N:=(d+1)(r-1)$ and let $a\colon\Delta_N\longrightarrow\R^d$ be an affine map.
	If $r=k\,q+s$, where $q=p^{\alpha}$ is a prime power with $k\geq1$ and $s\geq0$, then 
	\[
		\Tv(a)\ \geq\ \frac{q}{r!}\,\Big(\frac{q}{\alpha+1}\Big)^{\big\lceil\frac{(r-1-2s)(d+1)}{2}\big\rceil}.
	\]
\end{theorem}

\medskip
For $k=1$ and $s=0$, that is, for $r=q$ a prime power, Theorem~\ref{th : main} coincides with the bound of Vu\v{c}i\'c and \v{Z}ivaljevi\'c and of Hell, Theorem~\ref{th : VZ Hell}, since $\frac{q}{r!}=\frac{1}{(r-1)!}$ and $(r-1)(d+1)=N$.

\medskip
The new cases are $s\geq1$ or $k\geq2$, and they have no restriction on $r$.
Choosing for $q$ the largest prime power which divides $r$ we obtain, for every $r$, a bound of the same form as in Theorem~\ref{th : VZ Hell}.

\medskip
\begin{corollary}
	\label{cor : every r}
	Let $r\geq3$, $d\geq1$ be integers, $N:=(d+1)(r-1)$ and let $a\colon\Delta_N\longrightarrow\R^d$ be an affine map.
	If $b(r)=p^\alpha$ is the largest prime power which divides $r$, then
	\[
		\Tv(a)\ \geq\ \frac{b(r)}{r!}\,\Big(\frac{b(r)}{\alpha+1}\Big)^{\lceil \frac{N}2\rceil}.
	\]
\end{corollary}

\medskip
Since $b(r)\geq3$ for every $r\geq3$, we have $\frac{b(r)}{\alpha+1}>1$, and so for every fixed $r\geq3$ the number $\Tv_r(d)$ grows at least exponentially in $d$.
The best exponential rate is obtained by maximising over all decompositions $r=kq+s$.
For an integer $r\geq3$ we set
\[
	\beta_r:=\max\Big\{\Big(\frac{q}{\alpha+1}\Big)^{\frac{r-1-2s}{2}}\ :\ r=kq+s,\ q=p^\alpha\ \text{a prime power},\ k\geq1,\ s\geq0\Big\}.
\]
The maximum is taken over a finite set, since $q\leq r$.
We fix a decomposition $r=kq+s$ at which the maximum is attained, and we denote its prime power $q$ by $q_r$; for instance $q_{10}=5$, see Table~\ref{tab : rates}.
Taking $q=b(r)$ and $s=0$ in the definition, we have $\beta_r\geq\big (\tfrac{b(r)}{\alpha+1}\big)^{\frac{r-1}2}>1$.
Since $q\geq\alpha+1$ for every prime power $q=p^\alpha$, the base of the power in Theorem~\ref{th : main} is at least $1$, and therefore Theorem~\ref{th : main} implies the following corollary.

\medskip
\begin{corollary}
	\label{cor : rate}
	Let $r\geq3$, $d\geq1$ be integers, $N:=(d+1)(r-1)$ and let $a\colon\Delta_N\longrightarrow\R^d$ be an affine map.
	Then, with $\beta_r$ and $q_r$ as above,
	\[
		\Tv(a)\ \geq\ \frac{q_r}{r!}\,\beta_r^{\,d+1}.
	\]
	Consequently, for every fixed $r\geq3$ the minimal numbers $\Tv_r(d)$ satisfy
	\[
		1\ <\ \beta_r\ \leq\ \liminf_{d\to\infty}\,\Tv_r(d)^{\frac{1}{d}}\ \leq\ \limsup_{d\to\infty}\,\Tv_r(d)^{\frac{1}{d}}\ \leq\ (r-1)! .
	\]
\end{corollary}

\medskip
The upper bound $(r-1)!$ comes from the $(d+1)$-clouds configuration, and Sierksma's conjecture states that $\Tv_r(d)^{\frac{1}{d}}=(r-1)!$ for all $d$.
For large $r$ the rate $\beta_r$ is roughly the square root of $(r-1)!$.
Indeed, taking $k=1$ and for $q$ the largest prime with $q\leq r$, the prime number theorem gives $s=r-q=o(r)$, and therefore
\[
	\log\beta_r\ =\ \big(\tfrac12+o(1)\big)\,r\log r\ =\ \big(\tfrac12+o(1)\big)\log\,(r-1)!\qquad (r\to\infty);
\]
quantitatively, $s\leq r^{0.525}$ for all large $r$ by Baker, Harman and Pintz~\cite{BakerHarmanPintz2001}.
The upper estimate $\log\beta_r\leq\frac{r-1}{2}\log r$ holds since $\frac{q}{(\alpha+1)}<r$.
Thus, for every $r$, the bound of Corollary~\ref{cor : rate} is of the same quality as the bound of Vu\v{c}i\'c--\v{Z}ivaljevi\'c and Hell for prime powers.

\medbreak
The rates $\beta_r$ of the previous corollary are compared with the conjectured rate $(r-1)!$ in Table~\ref{tab : rates} for all $r\leq 30$ which are not prime powers.

 \begin{table}[ht]
	\small
	\begin{tabular}{@{}rcrr@{\qquad}rcrr@{}}
		\toprule
		$r$ & $(q,k,s)$ & $\beta_r$ & $(r-1)!$ & $r$ & $(q,k,s)$ & $\beta_r$ & $(r-1)!$\\
		\midrule
		$6$ & $(5,1,1)$ & $3.95$ & $120$ & $20$ & $(19,1,1)$ & $2.04\cdot10^{8}$ & $1.2\cdot10^{17}$\\
		$10$ & $(5,2,0)$ & $61.8$ & $3.6\cdot10^{5}$ & $21$ & $(19,1,2)$ & $6.63\cdot10^{7}$ & $2.4\cdot10^{18}$\\
		$12$ & $(11,1,1)$ & $2146$ & $4.0\cdot10^{7}$ & $22$ & $(11,2,0)$ & $5.94\cdot10^{7}$ & $5.1\cdot10^{19}$\\
		$14$ & $(13,1,1)$ & $2.96\cdot10^{4}$ & $6.2\cdot10^{9}$ & $24$ & $(23,1,1)$ & $1.37\cdot10^{11}$ & $2.6\cdot10^{22}$\\
		$15$ & $(13,1,2)$ & $1.16\cdot10^{4}$ & $8.7\cdot10^{10}$ & $26$ & $(25,1,1)$ & $3.89\cdot10^{10}$ & $1.6\cdot10^{25}$\\
		$18$ & $(17,1,1)$ & $9.35\cdot10^{6}$ & $3.6\cdot10^{14}$ & $28$ & $(27,1,1)$ & $2.32\cdot10^{10}$ & $1.1\cdot10^{28}$\\
		& & & & $30$ & $(29,1,1)$ & $4.77\cdot10^{15}$ & $8.8\cdot10^{30}$\\
		\bottomrule
	\end{tabular}\vspace{5pt}
	\caption{\small The exponential rate $\beta_r$ of Corollary~\ref{cor : rate} for $r\leq30$ not a prime power, the decomposition $r=kq+s$ at which it is attained, and Sierksma's conjectured rate $(r-1)!$.}
	\label{tab : rates}
\end{table}

\medbreak
Furthermore, Theorem~\ref{th : main} improves the known bounds for most prime powers which are not primes.
Let $r=p^\alpha$ be a prime power with $\alpha\geq2$.
The bound of Theorem~\ref{th : VZ Hell} grows with the exponential rate $(\frac{r}{\alpha+1})^{\frac{r-1}{2}}$ in $d$, while the bound of Theorem~\ref{th : main} grows with the rate $\beta_r$.
If $r$ is sufficiently large, or if $r\leq3000$ and $r\notin\{4,9,16,25,27,121\}$, then $\beta_r>(\frac{r}{\alpha+1})^{\frac{r-1}{2}}$, so that for all sufficiently large $d$ the bound of Theorem~\ref{th : main} exceeds the bound of Theorem~\ref{th : VZ Hell}.
For example, for $r=8$ the exponential rate improves from $11.3$ to $22.9$, for $r=32$ from $1.86\cdot10^{11}$ to $1.82\cdot10^{17}$, and for $r=64$ from $1.88\cdot10^{30}$ to $2.0\cdot10^{42}$.

\medskip
The lower bound of Theorem~\ref{th : main} is obtained by counting Tverberg partitions which satisfy additional constraints on pairs of vertices of the simplex $\Delta_N$.
These constraints give the following Tverberg type results, valid for every number of parts $r$.

\medskip
For $r=kq+s$, where $q=p^\alpha$ is a prime power, $k\geq1$ and $s\geq0$, let
\[
	\Lambda\ :=\ (\FF_p^{\,\alpha}\times[k])\ \sqcup\ \{\infty_1,\dots,\infty_s\},
\]
a set with $kq+s=r$ elements, and fix a bijection $[r]\longrightarrow\Lambda$.
In other words, the index set $[r]$ is split into $k$ blocks of $q$ indices each, the indices of every block are identified with the elements of the vector space $\FF_p^{\,\alpha}$, and the remaining $s$ indices are called $\infty_1,\dots,\infty_s$.
A Tverberg partition is then written as $(F_c : c\in\Lambda)$; as before it is an ordered collection, but now its faces are labelled by the elements of $\Lambda$.
For a face $F_{(u,\iota)}$ we call $u\in\FF_p^{\,\alpha}$ its {\em $\FF_p^{\,\alpha}$-label} and $\iota\in[k]$ its {\em block}.
By $e_1,\dots,e_\alpha$ we denote the standard basis of the $\FF_p$-vector space $\FF_p^{\,\alpha}$.

\medskip
The constraints are imposed on the pairs of an {\em oriented matching} $M$ on the vertex set of $\Delta_N$, that is, a set of ordered pairs $(v,v')$ of distinct vertices of $\Delta_N$ such that every vertex belongs to at most one pair.
In other words, $M$ is a matching in the usual sense together with an orientation of each of its edges.

\medskip
\begin{theorem}[Tverberg partitions with difference constraints]
	\label{th : difference}
	Let $r\geq2$, $d\geq1$ be integers, $N:=(d+1)(r-1)$ and let $a\colon\Delta_N\longrightarrow\R^d$ be an affine map.
	Suppose $r=kq+s$, where $q=p^\alpha$ is a prime power with $k\geq1$ and $s\geq0$, and identify the index set $[r]$ with $\Lambda=(\FF_p^{\,\alpha}\times[k])\sqcup\{\infty_1,\dots,\infty_s\}$.
	Let $\Gamma\subseteq\FF_p^{\,\alpha}$ be a set whose differences $\{g-g' : g,g'\in\Gamma\}$ generate the group $\FF_p^{\,\alpha}$, and let $M$ be an oriented matching on the vertex set of $\Delta_N$.
	Then there is a Tverberg partition $(F_c : c\in\Lambda)$ of the affine map $a$ whose faces contain all vertices of $\Delta_N$ and such that for every pair $(v,v')\in M$ the following holds: if $v\in F_{(u,\iota)}$ and $v'\in F_{(u',\iota')}$ for some $u,u'\in\FF_p^{\,\alpha}$ and some $\iota,\iota'\in[k]$, then $u'-u\in\Gamma$.
\end{theorem}

\medskip
The previous theorem says that for every pair $(v,v')\in M$ with neither vertex in one of the faces $F_{\infty_1},\dots,F_{\infty_s}$, the $\FF_p^{\,\alpha}$-label of the face containing $v'$ minus the $\FF_p^{\,\alpha}$-label of the face containing $v$ lies in $\Gamma$.
Here the blocks of the two faces are arbitrary, and a pair with a vertex in one of the faces $F_{\infty_1},\dots,F_{\infty_s}$ is not constrained.
For instance, for $q=p$ a prime, $k=1$, $s=0$ and $\Gamma=\{0,1\}$, the faces are $F_0,\dots,F_{p-1}$, and for every pair $(v,v')\in M$ the vertex $v'$ lies in the face of $v$ or in the next one, cyclically.
Figure~\ref{fig : difference} illustrates the theorem.
For $k=1$, $s=0$, that is for $r=q$ a prime power, the theorem holds for continuous maps $\Delta_N\longrightarrow\R^d$ as well.

\begin{figure}[ht]
\centering
\begin{tikzpicture}[
	cell/.style={draw,minimum width=16mm,minimum height=7mm,inner sep=0pt,font=\scriptsize},
	allowed/.style={cell,fill=black!22},
	forbidden/.style={cell,fill=white}]
\newcommand{\labelgrid}[4]{% #1 x-shift, #2 rows allowed in block 1 (0/1 flags), #3 same for block 2, #4 title
	\begin{scope}[xshift=#1cm]
	\node[font=\small] at (1.0,1.25) {#4};
	\node[font=\scriptsize] at (0.15,0.7) {block $1$};
	\node[font=\scriptsize] at (1.85,0.7) {block $2$};
	\foreach \u/\f in {#2}{\node[\f] at (0.15,-0.75*\u) {\ifnum\u=0\else$F_{(\u,1)}$\fi};}
	\foreach \u/\f in {#3}{\node[\f] at (1.85,-0.75*\u) {$F_{(\u,2)}$};}
	\foreach \u in {0,1,2}{\node[font=\scriptsize,anchor=east] at (-0.7,-0.75*\u) {$u=\u$};}
	\node[allowed] at (1.0,-2.45) {$F_{\infty_1}$};
	\node[cell,line width=1.3pt,fill=none] at (0.15,0) {$F_{(0,1)}\ni v$};
	\end{scope}}
\labelgrid{0}{0/allowed,1/allowed,2/forbidden}{0/allowed,1/allowed,2/forbidden}{$\Gamma=\{0,1\}$}
\labelgrid{5.0}{0/forbidden,1/allowed,2/allowed}{0/forbidden,1/allowed,2/allowed}{$\Gamma=\{1,2\}$}
\end{tikzpicture}

\smallskip
{\small (a)\quad $r=7=2\cdot3+1$: the faces of a Tverberg partition, arranged by block and $\FF_3$-label $u$, with $v\in F_{(0,1)}$.}

\bigskip
\begin{tikzpicture}[scale=0.95,
	pt/.style={circle,fill=black,inner sep=1.4pt},
	pr/.style={->,>=stealth,thick,shorten <=3pt,shorten >=3pt}]
	\coordinate (A) at (0,0); \coordinate (B) at (5,0); \coordinate (C) at (2.5,3.8);
	\coordinate (D) at (-0.5,1.6); \coordinate (E) at (5.5,0.8);
	\coordinate (G) at (1.7,3.2); \coordinate (H) at (3.3,-0.8);
	\fill[black!12] (A) -- (B) -- (C) -- cycle; \draw[thin] (A) -- (B) -- (C) -- cycle;
	\draw[line width=1.6pt] (D) -- (E);
	\draw[line width=1.6pt,dash pattern=on 5pt off 2.5pt] (G) -- (H);
	\node[circle,draw,inner sep=1.7pt,fill=white] at (2.5,1.2) {};
	\foreach \P in {A,B,C,D,E,G,H}{\node[pt] at (\P) {};}
	\node[below right=0pt,font=\scriptsize] at (B) {$a(v_1)$};
	\node[right=2pt,font=\scriptsize] at (E) {$a(v'_1)$};
	\node[left=2pt,font=\scriptsize] at (D) {$a(v_2)$};
	\node[left=2pt,font=\scriptsize] at (G) {$a(v'_2)$};
	\node[below=2pt,font=\scriptsize] at (H) {$a(v_3)$};
	\node[below left=0pt,font=\scriptsize] at (A) {$a(v'_3)$};
	\node[above=2pt,font=\scriptsize] at (C) {$a(w)$};
	\draw[pr,bend right=40] (B) to (E);
	\draw[pr,bend right=22] (D) to (G);
	\draw[pr,bend left=25] (H) to (A);
	\node[font=\small] at (1.45,0.45) {$a(F_0)$};
	\node[font=\small] at (4.75,1.2) {$a(F_1)$};
	\node[font=\small] at (3.85,-0.35) {$a(F_2)$};
\end{tikzpicture}

\smallskip
{\small (b)\quad $r=q=3$, $d=2$, $\Gamma=\{0,1\}$: an affine map $a\colon\Delta_6\longrightarrow\R^2$ and a Tverberg partition $(F_0,F_1,F_2)$ satisfying the constraint for $M=\{(v_1,v'_1),(v_2,v'_2),(v_3,v'_3)\}$.}
\caption{\small Theorem~\ref{th : difference}.
(a) If $v\in F_{(0,1)}$ (thick frame), then $v'$ may lie only in a shaded face: in a face whose $\FF_3$-label $u$ satisfies $u-0\in\Gamma$, in either block, or in $F_{\infty_1}$.
For $\Gamma=\{1,2\}$, which does not contain $0$, the faces with label $u=0$, including the face of $v$, are excluded; this is the mechanism behind Corollary~\ref{cor : matching}.
(b) The faces are $F_0=\conv\{v_1,v'_3,w\}$ (shaded triangle), $F_1=\conv\{v'_1,v_2\}$ (solid segment) and $F_2=\conv\{v_3,v'_2\}$ (dashed segment); their images meet in the marked point.
Each arrow points from $a(v)$ to $a(v')$ for a pair $(v,v')\in M$, and in each pair $v'$ lies in the face following the face of $v$ cyclically: $1-0$, $2-1$ and $0-2$ all equal $1\in\Gamma$.}
\label{fig : difference}
\end{figure}

\medskip
The assumption on $\Gamma$ cannot be dropped: for $\Gamma=\{0\}$, $k=1$ and $s=0$ the theorem would put the two vertices of every pair $(v,v')\in M$ into the same face.
This fails already for $d=1$.
Let $a\colon\Delta_{2r-2}\longrightarrow\R$ be injective on the $2r-1$ vertices.
In every Tverberg partition of $a$ at most one face contains a vertex mapped to the common point, and every other face contains vertices mapped to both sides of it; since there are only $2r-1$ vertices, one face is the vertex mapped to the median and every other face is an edge whose endpoints are mapped to opposite sides of the median.
Hence the two vertices mapped to the two smallest points always lie in different faces, and the theorem fails for $M$ consisting of this one pair.

\medskip
If $0\notin\Gamma$ and $s=0$, then $u'-u\in\Gamma$ implies $u'\neq u$, and so the two vertices of every pair $(v,v')\in M$ lie in different faces.
Such a set $\Gamma$ exists for every prime power $q\geq3$: for $\alpha=1$ take $\Gamma=\{1,2\}$, and for $\alpha\geq2$ take $\Gamma=\{e_1,\dots,e_\alpha,e_1+e_2\}$.
Since every $r\geq3$ has a prime power divisor $q\geq3$, for instance $q=b(r)$, Theorem~\ref{th : difference} with $k=r/q$ and $s=0$ yields the following result.

\medskip
\begin{corollary}[Affine matching theorem]
	\label{cor : matching}
	Let $r\geq3$, $d\geq1$ be integers, $N:=(d+1)(r-1)$ and let $a\colon\Delta_N\longrightarrow\R^d$ be an affine map.
	Then for every matching $M$ on the vertex set of $\Delta_N$ there is a Tverberg partition of $a$ whose faces contain all vertices of $\Delta_N$ and in which the two vertices of every edge of $M$ lie in different faces.
\end{corollary}

\medskip
Corollary~\ref{cor : matching} extends~\cite[Cor.\,1.7]{Blagojevic2025}, which treats a single pair for every $r$, to an arbitrary matching, and, for affine maps, Hell's matching theorem~\cite[Thm.\,2]{Hell2008constraints} from prime powers to all $r\geq3$.
For $r=2$ the statement is false: if $d=1$ and $a\colon\Delta_2\longrightarrow\R$ maps the vertices $w_1,w_2,w_3$ to points $a(w_1)<a(w_2)<a(w_3)$, then the only unordered Tverberg partition of $a$ is $\{\{w_2\},\conv\{w_1,w_3\}\}$, so the pair $\{w_1,w_3\}$ cannot be separated.
Correspondingly, for $q=2$ the only set $\Gamma\subseteq\FF_2$ with $0\notin\Gamma$ is $\{1\}$, and its differences do not generate $\FF_2$; so for $r=2$ Theorem~\ref{th : difference} applies only with $\Gamma=\FF_2$, where it is Tverberg's theorem.

\medskip
Theorem~\ref{th : difference} and Corollary~\ref{cor : matching} remain true for affine maps $\Delta_n\longrightarrow\R^d$ with $n\geq N$, and they need no general position assumption.
Like Theorem~\ref{th : main}, they are affine statements: for $r$ not a prime power each of them implies Theorem~\ref{th : Tverberg's original}, and therefore, by the counterexamples to the topological Tverberg conjecture, for such $r$ none of them extends to all continuous maps.

\medskip
Corollary~\ref{cor : matching} is the simplest instance of a more general question.
Hell~\cite{Hell2008constraints} introduced constraint graphs: a graph on the vertex set of $\Delta_N$ is a constraint graph if for every map $\Delta_N\longrightarrow\R^d$ there is a Tverberg partition in which any two adjacent vertices lie in different faces.
Such graphs are now called {\em Tverberg graphs}; see~\cite{Engstrom2011}.
In this language, Corollary~\ref{cor : matching} says that for every $r\geq3$ matchings are Tverberg graphs for affine maps.
For $r$ a prime power and continuous maps, Hell~\cite{Hell2008constraints} showed that matchings and several further classes of graphs, among them small complete graphs, stars, paths and cycles, are Tverberg graphs; see also Sch\"oneborn and Ziegler~\cite{SchoenebornZiegler2005}.
Engstr\"om~\cite{Engstrom2011} gave a local criterion for Tverberg graphs in terms of the neighbourhoods of the vertices, and Engstr\"om and Nor\'en~\cite{EngstromNoren2014} related Tverberg graphs to graph colouring.
Constraints given by more general graphs, such as forests and complete graphs, can be treated for every number of parts $r$ within the framework of this paper.
They will be the subject of a separate paper~\cite{BlagojevicTG}, which is in preparation.

\medskip
The lower bound of Theorem~\ref{th : main} can be improved in two ways.
First, the complexes used in its proof are invariant not only under translations of the $\FF_p^{\,\alpha}$-labels but also under permutations of the $k$ blocks and of the $s$ fixed labels, which multiplies the count by $k!\,s!$.
Second, the proof fixes one matching on the vertex set of $\Delta_N$ and uses, for every Tverberg partition, the least favourable position of its faces with fixed labels with respect to this matching; averaging over all matchings with $\lfloor\frac{N+1}2\rfloor$ pairs instead leads to an average which can be computed exactly.
For a polynomial $P(u)$ in one variable $u$ with real coefficients and an integer $j\geq0$ we denote by $[u^j]P$ the coefficient of $u^j$ in $P(u)$, so that $P(u)=\sum_{j\geq0}\big([u^j]P\big)\,u^j$; for example, $[u^1](1+u)^2=2$.
This notation goes back to Goulden and Jackson~\cite{GouldenJackson1983}.

\medskip
\begin{theorem}
	\label{th : averaged}
	Let $r\geq3$, $d\geq1$ be integers, $N:=(d+1)(r-1)$ and let $a\colon\Delta_N\longrightarrow\R^d$ be an affine map.
	If $r=k\,q+s$, where $q=p^{\alpha}$ is a prime power with $k\geq1$ and $s\geq0$, then
	\[
		\Tv(a)\ \geq\ \frac{q\,k!\,s!}{r!}\,\binom{N+1}{s(d+1)}\Big([u^{N+1-s(d+1)}]\,(1+u)^{N+1-2\lfloor\frac{N+1}2\rfloor}\Big(1+2u+\frac{\alpha+1}{q}\,u^2\Big)^{\lfloor\frac{N+1}2\rfloor}\Big)^{-1}.
	\]
\end{theorem}

\medskip
The bound of Theorem~\ref{th : averaged} is at least the bound of Theorem~\ref{th : main}, see Section~\ref{subsec : proof averaged}.
Its exponential rate follows from the estimate $[u^j]P\leq P(u)\,u^{-j}$, valid for every $u>0$ when all coefficients of $P$ are non-negative.
For an integer $r\geq3$ we set
\[
	\beta'_r:=\max_{r=kq+s}\ \sup_{u>0}\ \frac{(r-1)^{r-1}\,u^{r-1-s}}{s^s\,(r-1-s)^{r-1-s}\big(1+2u+\frac{\alpha+1}q\,u^2\big)^{\frac{r-1}2}},
\]
where the maximum is taken over all decompositions $r=kq+s$ with $q=p^\alpha$ a prime power, $k\geq1$ and $s\geq1$; the supremum is attained at the positive root of the equation
\[\frac{\alpha+1}{q}\,s\,u^2-(r-1-2s)\,u-(r-1-s)=0.\]

\medskip
\begin{corollary}
	\label{cor : averaged rate}
	For every integer $r\geq3$ the minimal numbers $\Tv_r(d)$ satisfy
	\[
		\liminf_{d\to\infty}\,\Tv_r(d)^{\frac{1}{d}}\ \geq\ \max\{\beta_r,\beta'_r\}.
	\]
\end{corollary}

\medskip
For instance, $\beta'_6\approx4.23$ and $\beta'_{15}\approx1.33\cdot10^{4}$, while Table~\ref{tab : rates} gives $\beta_6\approx3.95$ and $\beta_{15}\approx1.16\cdot10^{4}$.
Moreover $\beta'_{16}\approx6182$, attained at $r=1\cdot13+3$, exceeds the rate $(\frac{16}{5})^{\frac{15}2}\approx6146$ of Theorem~\ref{th : VZ Hell}, so by Corollary~\ref{cor : averaged rate} the number $16$ can be removed from the list of exceptional prime powers above.

\medskip
The remaining sections are organised as follows.
Sections~\ref{sec : test map} and~\ref{sec : fixed sphere} develop the framework, and Section~\ref{sec : proofs} contains the proofs of the results stated above.
In Section~\ref{sec : test map} we recall the test map, the zero-avoiding complex and the covering scheme from~\cite{Blagojevic2025}, and introduce the incidence space of a sub-complex of the deleted join.
Section~\ref{sec : fixed sphere} lets an elementary abelian group act on the labels; there we construct the equivariant Poincar\'e dual of a fixed point of the sphere and use it to prove the $G$-box criterion, Theorem~\ref{th : fixed sphere}, the topological core of the paper.
Finally, in Section~\ref{sec : proofs} the hypothesis of this criterion is verified for the complexes given by difference constraints, which proves Theorem~\ref{th : difference}, and a double counting of Tverberg partitions with translated difference complexes proves Theorem~\ref{th : main}; averaging this double counting over all matchings proves Theorem~\ref{th : averaged}.

%----------------------------------------------------------------------------------------------
%----------------------------------------------------------------------------------------------
\bigskip
\section{The test map and the incidence space}
\label{sec : test map}
\bigskip
%----------------------------------------------------------------------------------------------
%----------------------------------------------------------------------------------------------

In this section we recall the framework of zero-avoiding complexes, covering schemes and test maps developed in~\cite{Blagojevic2025} and rephrase it in a way which fits Tverberg type questions.
 For this we make some necessary preparations.

\medskip
Given a simplicial complex $X$ on a vertex set $V$ and a subset $U\subseteq V$, we denote by 
\[X[U]:=\{F\in X : F\subseteq U\}\]
the {\em induced sub-complex} of $X$ on $U$.

\medbreak
Given a field $\FF$ and $\tau\geq-1$, a space, or in particular a simplicial complex, $X$ is {\em homologically $\tau$-connected} over the field $\FF$ if $\widetilde H_i(X;\FF)=0$ for all $-1\leq i\leq\tau$. So in particular $X\neq\emptyset$.

\medskip
Let us fix integers $r\geq 2$ and $d\geq 1$, set $N:=(d+1)(r-1)$ as before, and let $n\geq0$ be an integer; in our applications only the case $n\geq N$ will be relevant.
Suppose $a\colon\Delta_n\longrightarrow\R^d$ is a fixed affine map.
The set of vertices of the simplex $\Delta_n$ is denoted by $\vertex(\Delta_n)=\{ p_0,\dots,p_n\}$.

%----------------------------------------------------------------------------------------------
\medskip
\subsection{Deleted join configuration space test map scheme}
\label{subsec : test map}
%----------------------------------------------------------------------------------------------
For this scheme we are considering the $r$-fold join $(\Delta_n)^{*r}$ of the simplex $\Delta_n$ and we denote the set of vertices of the $i$-th copy of $\Delta_n$ in the join  by $\{p_{0,i},\dots,p_{n,i}\}$.
The main configuration space is the $r$-fold $2$-wise deleted join:

\[
	\K:=(\Delta_n)^{*r}_{\Delta(2)}\ \cong\ [r]^{*(n+1)}=Q_0*\cdots*Q_n,
\]
where $Q_j:=\{p_{j,1},\dots,p_{j,r}\}$. For the isomorphism applied, see~\cite[Sec.\,6.3]{Matousek2008} and \cite[Lem.\,3.8.(1)]{BlagojevicZiegler2017}.
The complexes $\K$ and $(\Delta_n)^{*r}$ have the same vertex set $V:=\{p_{j,i} : 0\leq j\leq n,\ 1\leq i\leq r\}$.
A face of $\K$ is a set $J\subseteq V$ with $|J\cap Q_j|\leq1$ for all $j$, and it is the face $J=F_1*\cdots*F_r$, where $F_i$ is the face of $\Delta_n$ spanned by the vertices $p_j$ with $p_{j,i}\in J$.
Thus a face of $\K$ is a partial assignment of labels $i\in[r]$ to the vertices of $\Delta_n$, and a facet of $\K$ is an assignment of labels to all vertices.
A face $F_1*\cdots*F_r$ of $\K$ is called a {\em Tverberg face} of the affine map $a$ if $(F_1,\dots,F_r)$ is an ordered Tverberg partition of $a$.
Figure~\ref{fig : deleted join} illustrates these notions.

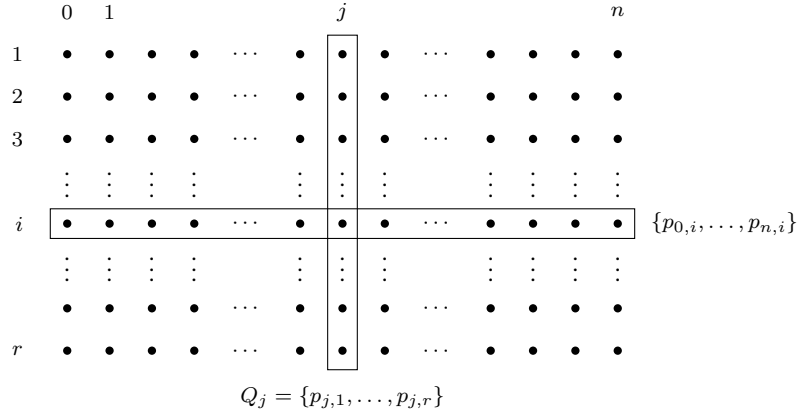
\begin{figure}[ht]
\centering
\begin{tikzpicture}[x=5.6mm,y=5.6mm,
	pt/.style={circle,fill=black,inner sep=1.1pt},
	lab/.style={font=\scriptsize}]
	\foreach \x in {0,1,2,3,5.5,6.5,7.5,10,11,12,13}{
		\foreach \y in {0,-1,-2,-4,-6,-7}{\node[pt] at (\x,\y) {};}
		\foreach \y in {-3,-5}{\node[lab] at (\x,\y+0.1) {$\vdots$};}}
	\foreach \y in {0,-1,-2,-4,-6,-7}{
		\node[lab] at (4.25,\y) {$\cdots$};
		\node[lab] at (8.75,\y) {$\cdots$};}
	\node[lab] at (0,1) {$0$}; \node[lab] at (1,1) {$1$};
	\node[lab] at (6.5,1) {$j$}; \node[lab] at (13,1) {$n$};
	\foreach \y/\l in {0/1,-1/2,-2/3,-4/i,-7/r}{\node[lab,anchor=east] at (-0.8,\y) {$\l$};}
	\draw (6.15,0.45) rectangle (6.85,-7.45);
	\draw (-0.4,-3.65) rectangle (13.4,-4.35);
	\node[lab] at (6.5,-8.1) {$Q_j=\{p_{j,1},\dots,p_{j,r}\}$};
	\node[lab,anchor=west] at (13.55,-4) {$\{p_{0,i},\dots,p_{n,i}\}$};
\end{tikzpicture}

\smallskip
{\small (a)\quad The vertex set $V$ of $\K$: the $j$-th column is the set $Q_j$, and the $i$-th row is the vertex set of the $i$-th copy of $\Delta_n$ in the join $(\Delta_n)^{*r}$.}

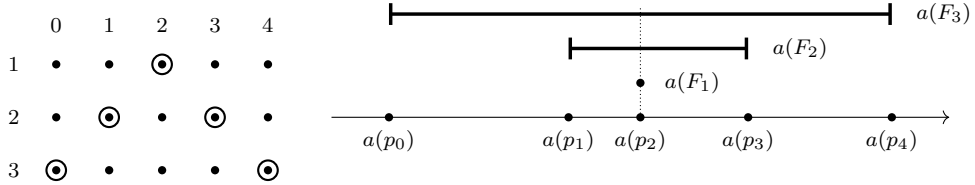
\bigskip
\begin{tikzpicture}[x=7mm,y=7mm,
	pt/.style={circle,fill=black,inner sep=1.1pt},
	sel/.style={circle,draw,thick,inner sep=2.6pt},
	lab/.style={font=\scriptsize}]
	\foreach \x in {0,...,4}{
		\node[lab] at (\x,0.75) {$\x$};
		\foreach \y in {1,2,3}{\node[pt] at (\x,-\y+1) {};}}
	\foreach \y in {1,2,3}{\node[lab,anchor=east] at (-0.5,-\y+1) {$\y$};}
	\foreach \x/\y in {0/3,1/2,2/1,3/2,4/3}{\node[sel] at (\x,-\y+1) {};}
	\begin{scope}[xshift=6.3cm,x=19mm,y=7mm]
		\draw[->] (-1.4,-1) -- (2.9,-1);
		\foreach \z/\j in {-1/0,0.25/1,0.75/2,1.5/3,2.5/4}{
			\node[pt] at (\z,-1) {};
			\node[lab,below=1pt] at (\z,-1) {$a(p_{\j})$};}
		\draw[densely dotted] (0.75,-1) -- (0.75,1.1);
		\node[pt] at (0.75,-0.35) {};
		\draw[line width=1.3pt,|-|] (0.25,0.3) -- (1.5,0.3);
		\draw[line width=1.3pt,|-|] (-1,0.95) -- (2.5,0.95);
		\node[lab,anchor=west] at (0.85,-0.35) {$a(F_1)$};
		\node[lab,anchor=west] at (1.6,0.3) {$a(F_2)$};
		\node[lab,anchor=west] at (2.6,0.95) {$a(F_3)$};
	\end{scope}
\end{tikzpicture}

\smallskip
{\small (b)\quad $r=3$, $d=1$, $n=N=4$: a facet $J$ of $\K$ (ringed vertices, one in each column), an affine map $a\colon\Delta_4\longrightarrow\R$, and the images $a(F_1)$, $a(F_2)$, $a(F_3)$ drawn above the line.}
\caption{\small The deleted join $\K=Q_0*\cdots*Q_n$.
(a) A set $J\subseteq V$ is a face of $\K$ if and only if it has at most one vertex in each column, and then $p_{j,i}\in J$ means that the vertex $p_j$ of $\Delta_n$ receives the label $i$.
(b) The facet $J=F_1*F_2*F_3$ with $F_1=\conv\{p_2\}$, $F_2=\conv\{p_1,p_3\}$ and $F_3=\conv\{p_0,p_4\}$.
Since $a(F_1)\cap a(F_2)\cap a(F_3)=\{a(p_2)\}$, it is a Tverberg face of $a$.}
\label{fig : deleted join}
\end{figure}

\medskip
Let $W_r:=\{(y_1,\dots,y_r)\in\R^r : y_1+\cdots+y_r=0\}$ and $\E:=W_r^{\oplus(d+1)}$.
Elements of $\E$ can be presented as $(d+1)\times r$ matrices $y=(y_{\kappa,i})_{0\leq\kappa\leq d,\,1\leq i\leq r}$ whose rows sum to zero.
In this case we assume the standard inner product given by $\langle y,y'\rangle=\sum_{\kappa,i}y_{\kappa,i}\,y'_{\kappa,i}$.
The unit sphere, with respect to this inner product, will be denoted by $S(\E)$.
Then $\dim(\E)=(d+1)(r-1)=N$ and $S(\E)\cong S^{N-1}$.

\medskip
The {\em test map} associated to the affine map $a\colon\Delta_n\longrightarrow
\R^d $ is the map $A\colon(\Delta_n)^{*r}\longrightarrow\E$ given by
\begin{equation}
	\label{eq : test map}
	t_1\cdot x_1+\cdots+t_r\cdot x_r\ \longmapsto\
	\Big(t_i\cdot \widehat{a}(x_i)-\frac1r\sum_{\ell=1}^rt_\ell\cdot\widehat{a}(x_\ell)\Big)_{1\leq i\leq r},
\end{equation}
or in expanded form by 
\begin{multline}
	\label{def of A}
	t_1\cdot x_1+\cdots+t_r\cdot x_r \ \longmapsto  \\
	\begin{pmatrix}
		t_1-\frac{1}{r},                                             & \dots\ , & t_r-\frac{1}{r}                                             \\[4pt]
		t_1\cdot a_1(x_1)- \frac{1}{r}\sum_{i=1}^rt_i\cdot a_1(x_i), & \dots\ , & t_r\cdot a_1(x_r)- \frac{1}{r}\sum_{i=1}^rt_i\cdot a_1(x_i) \\[4pt]
		\dots                                                        & \dots    & \dots                                                       \\[4pt]
		t_1\cdot a_d(x_1)- \frac{1}{r}\sum_{i=1}^rt_i\cdot a_d(x_i), & \dots\ , & t_r\cdot a_d(x_r)- \frac{1}{r}\sum_{i=1}^rt_i\cdot a_d(x_i)
	\end{pmatrix}.
\end{multline}
We used the short notation $\widehat{a}(x):=(1,a_1(x),\dots,a_d(x))^{\top}\in\R^{d+1}$, where $a_\kappa\colon\Delta_n\longrightarrow \R$ is the $\kappa$-th coordinate function of $a$.

\medskip
\begin{lemma}
	\label{lem : A is affine}
	The test map $A\colon(\Delta_n)^{*r}\longrightarrow\E$ is affine on the simplex $(\Delta_n)^{*r}$.
	In particular $A(F)=\conv(A(\vertex(F)))$ for every face $F$ of $(\Delta_n)^{*r}$.
\end{lemma}
\begin{proof}
	The barycentric coordinates of the point $t_1x_1+\cdots+t_rx_r$ of the simplex $(\Delta_n)^{*r}$ are $\mu_{j,i}=t_i\,x_{i,j}$, where $x_i=\sum_jx_{i,j}\,p_j$.
	Since $a$ is affine, $t_i\,\widehat a(x_i)=\sum_j\mu_{j,i}\,\widehat a(p_j)$, so every entry of~\eqref{eq : test map} is a linear function of the barycentric coordinates $\mu_{j,i}$.
\end{proof}

\medskip
For $1\leq i\leq r$ let $\varepsilon_i:=e_i-\frac1r(1,\dots,1)\in W_r$, and set $w_{j,i}:=A(p_{j,i})$.
With $y\in\E$ and $1\leq i\leq r$ we associate the affine function
\begin{equation}
	\label{eq : affine functions}
	f^y_i\colon\R^d\longrightarrow\R,\qquad z\longmapsto y_{0,i}+\sum_{\kappa=1}^dy_{\kappa,i}\,z_\kappa .
\end{equation}
In the following, $\delta_{i\ell}$ denotes the Kronecker delta.
\medskip
\begin{lemma}
	\label{lem : vertices of A}
	Let $1\leq i\leq r$, $0\leq j\leq n$, and $1\leq \ell\leq r$ be integers.
	\begin{compactenum}[\rm (1)]
		\item The $\ell$-th column of $w_{j,i}$ is $(\delta_{i\ell}-\frac1r)\,\widehat a(p_j)$, or in other words, $w_{j,i}=A(p_{j,i})=\varepsilon_i\otimes\widehat a(p_j)$. In particular $w_{j,i}\neq0$.
		\item For every $y\in\E$ it holds that $\langle y,w_{j,i}\rangle=f^y_i(a(p_j))$ and $f^y_1+\cdots+f^y_r=0$.
	\end{compactenum}
\end{lemma}
\begin{proof}
	(1) The claim follows from~\eqref{eq : test map} by taking $t_i=1$.
	(2) Since the rows of $y$ sum to zero, $\langle y,w_{j,i}\rangle=\sum_\kappa\widehat a(p_j)_\kappa\sum_\ell y_{\kappa,\ell}(\delta_{i\ell}-\frac1r)=\sum_\kappa\widehat a(p_j)_\kappa\,y_{\kappa,i}$, and $\sum_if^y_i=0$ for the same reason.
\end{proof}

\medskip
The next lemma is the classical deleted join configuration space test map scheme, see~\cite[Thm.\,3.9]{BlagojevicZiegler2017}, which is also sometimes called Sarkaria's tensor trick~\cite{Sarkaria1992}.

\medskip
\begin{lemma}[The classical configuration space test map scheme]
	\label{lem : CSTM}
	Suppose $F_1,\dots,F_r$ are pairwise disjoint faces of $\Delta_n$.
	Then $0\in A(F_1*\cdots*F_r)$ if and only if $a(F_1)\cap\cdots\cap a(F_r)\neq\emptyset$.
	Consequently, a face $J$ of $\K$ is a Tverberg face of $a$ if and only if $0\in\conv(A(J))$.
\end{lemma}
\begin{proof}
	By~\eqref{eq : test map}, $A(t_1x_1+\cdots+t_rx_r)=0$ if and only if $t_1\,\widehat a(x_1)=\dots=t_r\,\widehat a(x_r)$.
	Comparing first coordinates gives $t_1=\cdots=t_r=\frac1r$; in particular every face $F_i$ is non-empty, since a point of the join with $F_i=\emptyset$ has $t_i=0$.
	Then $\widehat a(x_1)=\cdots=\widehat a(x_r)$, that is, $a(x_1)=\cdots=a(x_r)$ with $x_i\in F_i$.
	Conversely, if $a(x_1)=\cdots=a(x_r)$ with $x_i\in F_i$, then $\frac1r x_1+\cdots+\frac1r x_r$ is mapped to $0$.
	The second claim follows from Lemma~\ref{lem : A is affine}.
\end{proof}

%----------------------------------------------------------------------------------------------
\medskip
\subsection{Zero-avoiding complexes}
\label{subsec : zero-avoiding}
%----------------------------------------------------------------------------------------------

Recall that $a\colon\Delta_n\longrightarrow\R^d$ is an affine map we have fixed and that $A\colon (\Delta_n)^{*r}\longrightarrow \E$ is its associated test map.
According to Lemma~\ref{lem : vertices of A}(1) we have that $0\notin A(V)$.
Now, to the affine map $A$ we associate the so-called {\em zero-avoiding complex}
\[
	\CC_A:=\{J\subseteq V : 0\notin\conv(A(J))\}
\]
which is a simplicial complex on the vertex set $V$.
For more details on this construction consult~\cite[Sec.\,2.1]{Blagojevic2025}.
From Lemma~\ref{lem : CSTM} it follows that an arbitrary sub-complex $\LL$ of $\K$ contains no Tverberg face of $a$ if and only if $\LL\subseteq\CC_A$.

\medskip
Next, for $y\in S(\E)$ let
\[
	U_{a,y}:=\{v\in V : \langle y,A(v)\rangle\leq 0\}
\]
be the set of vertices in $V$ on which the linear functional $\langle y,\cdot\rangle$ is non-positive.
The index $a$ records that this set depends not only on $y$ but also, through the test map $A$, on the affine map $a$.

\medskip
\begin{lemma}
	\label{lem : locality}
	Let $y\in S(\E)$ and $0\leq j\leq n$.
	Then
	\[
		U_{a,y}\cap Q_j=\{p_{j,i} : f^y_i(a(p_j))\leq0\}\neq\emptyset .
	\]
	In particular, $U_{a,y}\cap Q_j$ depends only on the values of the $r$ affine functions $f^y_1,\dots,f^y_r$ at the single point $a(p_j)$.
\end{lemma}
\begin{proof}
	The description of $U_{a,y}\cap Q_j$ follows from Lemma~\ref{lem : vertices of A}(2). 
	The set $U_{a,y}\cap Q_j$ is non-empty because the values $f^y_i(a(p_j))$ sum up to zero.
\end{proof}

\medskip
Figure~\ref{fig : locality} illustrates Lemma~\ref{lem : locality} for $d=1$.
For every $d$, the vertex $p_{j,i}$ lies in $U_{a,y}$ if and only if $a(p_j)$ lies in the set $\{z\in\R^d : f^y_i(z)\leq0\}$, which is a closed half-space unless $f^y_i$ is constant, so $U_{a,y}$ records the position of the points $a(p_0),\dots,a(p_n)$ with respect to the zero sets of $f^y_1,\dots,f^y_r$.
For $d=1$ the zero set of a non-constant $f^y_i$ is a single point of $\R$, and the upper part of the figure shows the graphs of the affine functions $f^y_1,\dots,f^y_r$ over the line $\R$, drawn as lines in the plane. The point $a(p_j)$ lies in the half-line $\{z\in\R : f^y_i(z)\leq0\}$ if and only if the graph of $f^y_i$ is on or below the horizontal axis at $a(p_j)$.
The set $U_{a,y}$ itself is displayed in the lower part of the figure: the column $Q_j$ is drawn below the point $a(p_j)$, and its vertex $p_{j,i}$ is filled if and only if the graph of $f^y_i$ is on or below the axis at $a(p_j)$, that is, if and only if $p_{j,i}\in U_{a,y}$.

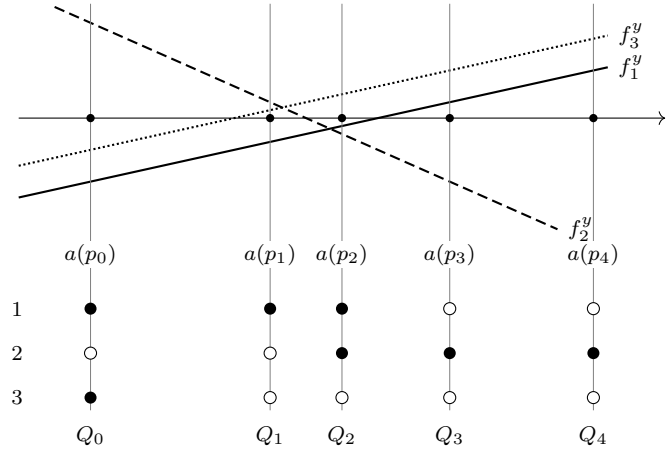
\begin{figure}[ht]
\centering
\begin{tikzpicture}[x=19mm,y=4.2mm,
	pt/.style={circle,fill=black,inner sep=1.1pt},
	hollow/.style={circle,draw,fill=white,inner sep=1.1pt},
	lab/.style={font=\scriptsize}]
	\draw[->] (-1.5,0) -- (3.0,0);
	\draw[thick] (-1.5,-2.5) -- (2.6,1.6) node[lab,anchor=west] {$f^y_1$};
	\draw[thick,dash pattern=on 4pt off 2pt] (-1.25,3.5) -- (2.25,-3.5) node[lab,anchor=west] {$f^y_2$};
	\draw[thick,densely dotted] (-1.5,-1.5) -- (2.6,2.6) node[lab,anchor=west] {$f^y_3$};
	\foreach \z/\j in {-1/0,0.25/1,0.75/2,1.5/3,2.5/4}{
		\draw[very thin,gray] (\z,3.6) -- (\z,-9.2);
		\node[pt] at (\z,0) {};
		\node[lab,fill=white,inner sep=1pt] at (\z,-4.3) {$a(p_{\j})$};}
	\foreach \y in {1,2,3}{\node[lab,anchor=east] at (-1.4,-4.6-1.4*\y) {$\y$};}
	\foreach \z/\a/\b/\c/\j in {-1/1/0/1/0,0.25/1/0/0/1,0.75/1/1/0/2,1.5/0/1/0/3,2.5/0/1/0/4}{
		\foreach \i/\m in {1/\a,2/\b,3/\c}{
			\ifnum\m=1 \node[pt,inner sep=1.6pt] at (\z,-4.6-1.4*\i) {};
			\else \node[hollow,inner sep=1.6pt] at (\z,-4.6-1.4*\i) {}; \fi}
		\node[lab,fill=white,inner sep=1pt] at (\z,-10) {$Q_{\j}$};}
\end{tikzpicture}
\caption{\small Lemma~\ref{lem : locality} for $r=3$, $d=1$, $n=4$ and the affine map $a$ of Figure~\ref{fig : deleted join}(b).
Up to a positive factor, which changes neither the functions' signs nor $U_{a,y}$, the point $y\in S(\E)$ is the matrix with rows $(-1,1,0)$ and $(1,-2,1)$, and so $f^y_1(z)=z-1$, $f^y_2(z)=1-2z$ and $f^y_3(z)=z$.
Below the point $a(p_j)$ the column $Q_j$ is drawn, and $p_{j,i}\in U_{a,y}$ (filled) if and only if $f^y_i(a(p_j))\leq0$.
Since $f^y_1+f^y_2+f^y_3=0$, at every point at least one of the functions is non-positive, and so every column meets $U_{a,y}$.}
\label{fig : locality}
\end{figure}

%----------------------------------------------------------------------------------------------
\medskip
\subsection{The covering scheme}
\label{subsec : covering scheme}
%----------------------------------------------------------------------------------------------
The sets $U_{a,y}$ we just introduced appear in a natural way in the covering scheme of~\cite[Sec.\,3]{Blagojevic2025}, which we now recall.
While the scheme is not directly used in the proofs of the main results, it does explain the role of the sets $U_{a,y}$ and motivates the incidence space which follows in  Section~\ref{subsec : incidence space}.

For a simplicial complex $X$ on the vertex set $V$ the simplicial complex defined by
\[
	X^*:= \{ F\subseteq V :V-F \notin X\}
\]
is called the {\em Alexander dual} of $X$.
Moreover, $(X^*)^*=X$.
For the case of the zero-avoiding complex $\CC_A$ of the affine map $A$ we have that the corresponding Alexander dual complex is given by
\begin{equation}\label{def of dual of zero-avoiding complex}
\CC_A^*:=\{F\subseteq V : V-F\notin\CC_A\}=\{F\subseteq V : 0\in\conv(A(V-F))\}.
\end{equation} 
With this definition, the combinatorial version of the Alexander duality says the following, see~\cite[Thm.\,1.1]{BjornerTancer2009}.
\begin{theorem}
	\label{th : combinatorial Alexander duality}
	For a simplicial complex $X$ on the finite vertex set $V$ which is not the simplex $\Delta_V$ there are isomorphisms
	\begin{equation}
		\label{iso : Alexander 01}
		\widetilde{H}_i(X;\FF)\cong  \widetilde{H}^{|V|-3-i}(X^*;\FF)
		\qquad\text{and}\qquad
		\widetilde{H}_i(X^*;\FF)\cong  \widetilde{H}^{|V|-3-i}(X;\FF)
	\end{equation}
	for every $i\in\Z$ and any field $\FF$.
\end{theorem}

\medskip
Now, let $\LL\subseteq\K$ be a sub-complex of the deleted join.
Every face of the join $\CC_A^**\LL$ is of the form $I*J$ with $I\in\CC_A^*$ and $J\in\LL$, where $I$ and $J$ are subsets of two disjoint copies of $V$.
If $\LL$ contains no Tverberg face of the affine map $a$, then
\begin{equation}
	\label{eq : covering scheme}
	\CC_A^**\LL\ =\ \bigcup_{v\in V}\ \CC_A^*[V-\{v\}]*\LL[V-\{v\}].
\end{equation}
This equality is the so-called covering scheme of~\cite[Lem.\,3.2]{Blagojevic2025}.
Indeed, identifying both copies of $V$ with $V$, a face $I*J$ which lies in none of these sub-complexes satisfies $I\cup J=V$.
Then $V-I\subseteq J$, and $I\in\CC_A^*$ gives $0\in\conv(A(V-I))\subseteq\conv(A(J))$, so $J\in\LL$ is a Tverberg face by Lemma~\ref{lem : CSTM}.
For $\emptyset\neq U\subsetneq V$ the intersection of the members of the cover~\eqref{eq : covering scheme} indexed by (missing) vertices $v\in V-U$ is the join $\CC_A^*[U]*\LL[U]$.

\medskip
Now, for finite simplicial complexes $K_1$ and $K_2$ and a field $\FF$ we have the K\"unneth formula for joins
\begin{equation}
	\label{eq : Kunneth for joins}
	\widetilde H_{i+1}(K_1*K_2;\FF)\cong\bigoplus_{i_1+i_2=i}\widetilde H_{i_1}(K_1;\FF)\otimes\widetilde H_{i_2}(K_2;\FF).
\end{equation}
By this formula the join $\CC_A^*[U]*\LL[U]$ is acyclic, that is, has trivial reduced homology over $\FF$, whenever $\CC_A^*[U]$ is acyclic.
We now explain that this is the case for all subsets $\emptyset\neq U\subsetneq V$ except those of the form $U_{a,y}$ for some $y\in S(\E)$.
In~\cite[Sec.\,2]{Blagojevic2025} the vectors $A(v)$ are assumed to be unit vectors; normalising them changes neither $\CC_A$, $\CC_A^*$ nor the sets $U_{a,y}$, so the results of~\cite{Blagojevic2025} quoted below apply.
Following~\cite[Sec.\,2.4]{Blagojevic2025}, for $U\subseteq V$ consider
\begin{align*}
	N(V,U)&:=\{y\in S(\E) : \langle y,A(v)\rangle>0\ \text{for}\ v\in V-U\ \text{and}\ \langle y,A(u)\rangle\leq0\ \text{for}\ u\in U\}\\
	&\phantom{:}=\{y\in S(\E) : U_{a,y}=U\}.
\end{align*}
Let $\emptyset\subsetneq U\subsetneq V$.
If $0\in\conv(A(V-U))$, then every subset of $U$ lies in $\CC_A^*$, and $\CC_A^*[U]$ is a simplex~\cite[Lem.\,2.12]{Blagojevic2025}.
Otherwise, the open set $B:=\{y\in S(\E) : \langle y,A(v)\rangle>0\ \text{for}\ v\in V-U\}$ is a non-empty open ball of dimension $N-1$, and the Alexander dual of $\CC_A^*[U]$ inside the boundary of the simplex on $U$ is homotopy equivalent to $B-N(V,U)$~\cite[Lem.\,2.14]{Blagojevic2025}.
Hence, if $N(V,U)=\emptyset$, this dual is contractible and $\CC_A^*[U]$ is acyclic.
In other words, $\CC_A^*[U]$ can have non-trivial reduced homology only if $U=U_{a,y}$ for some $y\in S(\E)$, and under further assumptions on $U$ its homology is computed from the topology of $N(V,U)$ in~\cite[Lem.\,2.21]{Blagojevic2025}.
So only the sets $U_{a,y}$ contribute to the covering scheme, and each of them only through the complex $\LL[U_{a,y}]$.
By Lemma~\ref{lem : locality} these sets are local: $U_{a,y}$ meets every column $Q_j$, and $U_{a,y}\cap Q_j$ depends only on the values of the functions $f^y_i$ at the point $a(p_j)$.

\medskip
As an illustration, we show how the covering scheme yields Tverberg's theorem, Theorem~\ref{th : Tverberg's original}, for every $n\geq N$.
Suppose that $\K$ contains no Tverberg face of $a$, that is, $\LL=\K$ in~\eqref{eq : covering scheme}.
Then the $|V|$-element family of sub-complexes $\{\CC_A^*[V-\{v\}]*\K[V-\{v\}] : v\in V\}$ covers $\CC_A^**\K$.
Recall that $0=\frac1r(w_{j,1}+\cdots+w_{j,r})\in\conv(A(Q_j))$ for every $j$, by Lemma~\ref{lem : vertices of A}(1).
For every proper subset of the index set $V$, the intersection of the corresponding members is $\CC_A^*[U]*\K[U]$ with $U\neq\emptyset$.
It is non-empty, since every $u\in U$ is a vertex of $\CC_A^*$.
Indeed, the set $V-\{u\}$ contains a whole column $Q_j$, and so $0\in\conv(A(V-\{u\}))$.
The intersection of all members is empty.
Hence the nerve of the cover is the boundary of the simplex on the vertex set $V$, a sphere of dimension $|V|-2$.
We derive a contradiction from the homological Nerve theorem of Meshulam~\cite[Thm.\,2.1]{Meshulam2001}: if a connected finite simplicial complex $X$ is covered by sub-complexes such that every non-empty intersection $X'$ of $t\geq1$ members satisfies $\widetilde H_j(X';\FF)=0$ for $0\leq j\leq k-t+1$, then $\widetilde H_j(X;\FF)$ is isomorphic to the $j$-th reduced homology of the nerve of the cover for $0\leq j\leq k$.

\medskip
The connectivity estimates rest on two facts and on the K\"unneth formula, which implies that the join of a homologically $s$-connected and a homologically $t$-connected complex is homologically $(s+t+2)$-connected.
First, by~\cite[Cor.\,2.16]{Blagojevic2025}, for every $U\subsetneq V$ the complex $\CC_A^*[U]$ is homologically $(|U|-N-2)$-connected, while by~\cite[Lem.\,2.11]{Blagojevic2025} the complex $\CC_A^*=\CC_A^*[V]$ is only homologically $(|V|-N-3)$-connected.
Indeed, for $U\subsetneq V$ with $0\notin\conv(A(V-U))$ the Alexander dual of $\CC_A^*[U]$ inside the boundary of the simplex on $U$ is homotopy equivalent to an open subset of the ball $B$ of dimension $N-1$, while for $U=V$ the ball $B$ is replaced by the whole sphere $S(\E)$, and its top cohomology may survive.
Second, if $U$ meets every column $Q_j$, then $\K[U]=(U\cap Q_0)*\cdots*(U\cap Q_n)$ is a join of $n+1$ non-empty $0$-dimensional complexes, and so it is homologically $(n-1)$-connected; for $U=V$ see also~\cite[Sec.\,3.2.2]{BlagojevicZiegler2017}.
Now let $\emptyset\subsetneq U\subsetneq V$.
If $0\in\conv(A(V-U))$, then $\CC_A^*[U]$ is a simplex and the join $\CC_A^*[U]*\K[U]$ is contractible.
Otherwise $U$ meets every column $Q_j$, because $0\in\conv(A(Q_j))$ for every $j$, and the two facts we just discussed show that $\CC_A^*[U]*\K[U]$ is homologically $(|U|+n-N-1)$-connected.
An intersection of $t$ members has $|U|=|V|-t$, so the hypothesis of the Nerve theorem holds with $k:=|V|+n-N-2$.
For $U=V$ the same two facts show that $\CC_A^**\K$ is homologically $(|V|+n-N-2)$-connected, and in particular connected.
Because $n\geq N$, we have that $k\geq|V|-2$, and the Nerve theorem gives 
\[\widetilde H_{|V|-2}(\CC_A^**\K;\FF)\cong\widetilde H_{|V|-2}(S^{|V|-2};\FF)\cong\FF.\]
This contradicts the vanishing of $\widetilde H_{|V|-2}(\CC_A^**\K;\FF)$, because $|V|-2\leq|V|+n-N-2$.
Consequently, $\K$ contains a Tverberg face of $a$.
This completes the proof of Theorem~\ref{th : Tverberg's original}.

\medskip
For an arbitrary sub-complex $\LL\subseteq\K$ the same argument applies analogously, except that the connectivity of $\K[U]$ is replaced by that of $\LL[U]$.
By the discussion above it suffices to consider $U=V$ and the sets $U=U_{a,y}$, for which $0\notin\conv(A(V-U))$.
Thus, if $\LL$ and all the complexes $\LL[U_{a,y}]$, $y\in S(\E)$, are homologically $(n-1)$-connected over $\FF$, then $\LL$ contains a Tverberg face of $a$.
For the difference constraints of Section~\ref{sec : difference complexes} this condition fails in general.

\medskip
The incidence space introduced in the next subsection collects the complexes $\LL[U_{a,y}]$ over the sphere $S(\E)$ itself, as the fibres of a projection to $S(\E)$, instead of over the finite family of sets $U$.
It has two features which the covering scheme lacks and which we will utilise.
First, if $\LL$ contains no Tverberg face of $a$, then the incidence space has the cohomology of $\LL$ itself, Proposition~\ref{prop : projection to L}.
Second, a group acting on the labels acts on the incidence space, and over the fixed sphere of this action the fibres are the complexes $\LL[U]$ for a restricted family of sets $U$, the $G$-boxes of Section~\ref{subsec : group actions}; this is what makes the $G$-box criterion, Theorem~\ref{th : fixed sphere}, possible.

%----------------------------------------------------------------------------------------------
\medskip
\subsection{The incidence space}
\label{subsec : incidence space}
%----------------------------------------------------------------------------------------------
Let $\LL\subseteq\K$ be a sub-complex of the deleted join.
For a point $\lambda=\sum_{v}\lambda_v\,v$ of  $\LL$, written in barycentric coordinates, let $\supp(\lambda)=\{v : \lambda_v>0\}$ be the (unique) face whose relative interior contains $\lambda$.

\medskip
\begin{definition}
	\label{def : incidence space}
	The {\em incidence space} of a sub-complex $\LL\subseteq\K$ with respect to the affine map $a\colon\Delta_n\longrightarrow\R^d$ is
	\[
		T_{a,\LL}:=\{(y,\lambda)\in S(\E)\times\LL : \supp(\lambda)\subseteq U_{a,y}\},
	\]
	with the restrictions of the natural projections  $\pr_1\colon T_{a,\LL}\longrightarrow S(\E)$ and $\pr_2\colon T_{a,\LL}\longrightarrow\LL$.
\end{definition}

\medskip
The fibre of $\pr_1$ over $y\in S(\E)$ is the induced sub-complex $\LL[U_{a,y}]$.
Over $\lambda\in\LL$ the fibre of $\pr_2$ is the closed subspace 
$
	\{y\in S(\E) : \langle y,A(v)\rangle\leq 0\ \text{for all}\ v\in\supp(\lambda)\}
$
of the sphere $S(\E)$.

\medskip
\begin{lemma}
	\label{lem : T is compact}
	The incidence space $T_{a,\LL}$ is a closed, hence compact, subset of $S(\E)\times\LL$.
\end{lemma}
\begin{proof}
	The condition $\supp(\lambda)\subseteq U_{a,y}$ says that $\lambda_v=0$ or $\langle y,A(v)\rangle\leq0$ for every vertex $v\in V$, that is, $\lambda_v\cdot\max\{0,\langle y,A(v)\rangle\}=0$.
	These functions are continuous on $S(\E)\times\LL$ and so $T_{a,\LL}$ is a closed subset of $S(\E)\times\LL$.
\end{proof}

\medskip
Now, we recall the classical Vietoris--Begle mapping theorem \cite[Thm.\,6.9.15]{Spanier1966} because it will be used on several occasions.
For that reason, from now on, we assume Alexander--Spanier cohomology~\cite[Ch.\,6]{Spanier1966}.
In the case of paracompact and locally contractible spaces the Alexander--Spanier cohomology coincides with singular cohomology~\cite[Cor.\,6.9.5]{Spanier1966}.

\medskip
\begin{theorem}[Vietoris--Begle mapping theorem]
	\label{lem : Vietoris-Begle}
	Let $f\colon X\longrightarrow Y$ be a closed continuous surjection of paracompact Hausdorff spaces, let $\FF$ be a field and $\tau\geq 0$, and assume that $\widetilde H^i(f^{-1}(y);\FF)=0$ for all $y\in Y$ and all $i\leq\tau$.
	Then the induced map in cohomology $f^*\colon H^i(Y;\FF)\longrightarrow H^i(X;\FF)$ is an isomorphism for $i\leq\tau$ and a monomorphism for $i=\tau+1$.
\end{theorem}

\medskip
Utilising the Vietoris--Begle mapping theorem we show that the incidence space of a sub-complex $\LL\subseteq\K$ without Tverberg faces of $a$ has the cohomology of $\LL$.

\medskip
\begin{proposition}
	\label{prop : projection to L}
	If the sub-complex $\LL\subseteq\K$ contains no Tverberg face of the affine map $a\colon\Delta_n\longrightarrow\R^d$, then all fibres of $\pr_2\colon T_{a,\LL}\longrightarrow\LL$ are non-empty, compact and contractible.
	Consequently, the induced map in cohomology  $\pr_2^*\colon H^*(\LL;\FF)\longrightarrow H^*(T_{a,\LL};\FF)$ is an isomorphism.
\end{proposition}
\begin{proof}
	Let $\lambda\in\LL$ and let $\sigma=\supp(\lambda)$.
	Since $\sigma$ is not a Tverberg face, $0\notin\conv(A(\sigma))$ by Lemma~\ref{lem : CSTM}, and a hyperplane strictly separating $0$ from the compact convex set $\conv(A(\sigma))$ gives $y_0\in S(\E)$ with $\langle y_0,A(v)\rangle<0$ for all $v\in\sigma$.
	So the fibre $\pr_2^{-1}(\lambda)$ contains $y_0$.
	For arbitrary $y$ in the fibre and $t\in[0,1]$ it holds that
	\[
	\langle(1-t)y+ty_0,A(v)\rangle=(1-t)\langle y ,A(v)\rangle+t\langle y_0,A(v)\rangle \leq 0.
	\]
	Also $(1-t)y+ty_0\neq 0$ because otherwise $y=-\frac t{1-t}\,y_0$ would be positive on $A(\sigma)\neq\emptyset$ and $y$ would not belong to the fibre of $\pr_2 $.
	Normalising these vectors gives a deformation retraction of the fibre $\pr_2^{-1}(\lambda)$ onto $y_0$.
	The fibre is closed in $S(\E)$, hence compact.
	Now $\pr_2$ is a closed surjection of compact spaces with contractible fibres, and Theorem~\ref{lem : Vietoris-Begle} applies in all degrees saying that the induced map in cohomology  $\pr_2^*\colon H^*(\LL;\FF)\longrightarrow H^*(T_{a,\LL};\FF)$ is an isomorphism.
\end{proof}

%----------------------------------------------------------------------------------------------
%----------------------------------------------------------------------------------------------
\bigskip
\section{The Poincar\'e dual of a fixed point and the $G$-box criterion}
\label{sec : fixed sphere}
\bigskip
%----------------------------------------------------------------------------------------------
%----------------------------------------------------------------------------------------------

In this section we discuss group actions on the set of labels $[r]$ and develop further methods for the proofs of the main results. 
The highlights of this section are the construction of the equivariant Poincar\'e dual of a fixed point in Lemma~\ref{lem : Thom class} and the $G$-box criterion in Theorem~\ref{th : fixed sphere}.

%----------------------------------------------------------------------------------------------
\medskip
\subsection{Groups acting on the set of labels}
\label{subsec : group actions}
%----------------------------------------------------------------------------------------------
Let $\sym_r$ be the symmetric group on $r$ letters, or in other words the group of all bijections $[r]\to [r]$.
It naturally acts on the set of labels $[r]$, and so does each of its subgroups $G$.
This action extends to an action on the vertex set $V$ of $(\Delta_n)^{*r}$ by $g\cdot p_{j,i}:=p_{j,g(i)}$, and hence on the simplicial complexes $(\Delta_n)^{*r}$ and $\K$.
Furthermore, $\sym_r$, and so every subgroup $G$, acts on $\E$ by permuting the columns, $(g\cdot y)_{\kappa,i}:=y_{\kappa,g^{-1}(i)}$ for $y\in\E$.
The scalar product on $\E$ is invariant, and by~\eqref{eq : test map} the test map associated to an affine map $a\colon\Delta_n\longrightarrow\R^d$ is equivariant, that is $A(g\cdot\lambda)=g\cdot A(\lambda)$.
Consequently $\CC_A$ is $G$-invariant and $U_{a,g\cdot y}=g\cdot U_{a,y}$.
If $\LL\subseteq\K$ is $G$-invariant, then $T_{a,\LL}$ is invariant under the diagonal action $g\cdot(y,\lambda):=(g\cdot y,g\cdot\lambda)$, and so both projections $\pr_1$ and $\pr_2$ are equivariant.

\medskip
An {\em elementary abelian group} is a group isomorphic to $(\Z/p)^\alpha$ for a prime $p$ and some $\alpha\geq0$.
Let $G\leq\sym_r$ be an elementary abelian subgroup of $\sym_r$.
Since $G$ is abelian and transitive on each of its orbits in $[r]$, all elements of an orbit have the same stabiliser, and the size of every orbit has to be a power of $p$.

\medskip
Now, identifying each $Q_j$ with $[r]$ via $p_{j,i}\mapsto i$, we call a set $U\subseteq V$ a {\em $G$-box} if every $U\cap Q_j$ is a non-empty union of $G$-orbits.
For example, $V$ itself is a $G$-box, and Figure~\ref{fig : G-box} shows another example.

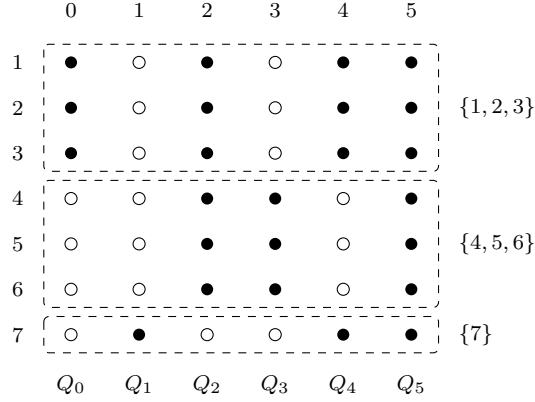
\begin{figure}[hb]
\centering
\begin{tikzpicture}[x=9mm,y=6mm,
	pt/.style={circle,fill=black,inner sep=1.6pt},
	hollow/.style={circle,draw,fill=white,inner sep=1.6pt},
	lab/.style={font=\scriptsize}]
	\foreach \a/\b in {1/3,4/6,7/7}{
		\draw[dashed,rounded corners=2pt] (-0.4,-\a+0.4) rectangle (5.4,-\b-0.4);}
	\node[lab,anchor=west] at (5.55,-2) {$\{1,2,3\}$};
	\node[lab,anchor=west] at (5.55,-5) {$\{4,5,6\}$};
	\node[lab,anchor=west] at (5.55,-7) {$\{7\}$};
	\foreach \i in {1,...,7}{\node[lab,anchor=east] at (-0.55,-\i) {$\i$};}
	\foreach \j/\A/\B/\C in {0/1/0/0,1/0/0/1,2/1/1/0,3/0/1/0,4/1/0/1,5/1/1/1}{
		\node[lab] at (\j,0.15) {$\j$};
		\node[lab] at (\j,-8.1) {$Q_{\j}$};
		\foreach \i in {1,2,3}{\ifnum\A=1 \node[pt] at (\j,-\i) {}; \else \node[hollow] at (\j,-\i) {}; \fi}
		\foreach \i in {4,5,6}{\ifnum\B=1 \node[pt] at (\j,-\i) {}; \else \node[hollow] at (\j,-\i) {}; \fi}
		\ifnum\C=1 \node[pt] at (\j,-7) {}; \else \node[hollow] at (\j,-7) {}; \fi}
\end{tikzpicture}
\caption{\small A $G$-box $U$ (filled vertices) for $r=7$, $n=5$, and the group $G\cong\Z/3$ generated by the permutation $(1\,2\,3)(4\,5\,6)$ of $[7]$.
The rows of the $G$-orbits $\{1,2,3\}$, $\{4,5,6\}$ and $\{7\}$ are framed.
Every column $U\cap Q_j$ is a non-empty union of orbits: each frame is either filled or empty within a column, and no column is empty.
By Lemma~\ref{lem : fibres over the fixed sphere}, every set $U_{a,y}$ with $y\in S(\E^G)$ is of this form.}
\label{fig : G-box}
\end{figure}

\medskip
\begin{lemma}
	\label{lem : decomposition}
	Let $G\leq\sym_r$ be an elementary abelian group.
	\begin{compactenum}[\rm (1)]
		\item The fixed point subspace $\E^G$ consists of the matrices $y\in\E$ whose columns $y_{*,i}$ and $y_{*,i'}$ coincide whenever $i$ and $i'$ lie in one $G$-orbit, and $\dim(\E^G)=(d+1)(|[r]/G|-1)$, where $|[r]/G|$ is the number of $G$-orbits.
		\item The orthogonal complement $\Eperp$ is $G$-invariant and has no non-zero fixed vector.
		\item $S(\E^G)$ is the fixed point set of the (induced) action of $G$ on $S(\E)$.
	\end{compactenum}
\end{lemma}
\begin{proof}
	A vector of $\R^r$ is $G$-fixed if and only if it is constant on every orbit; intersecting with $W_r$ gives a space of dimension $|[r]/G|-1$, and $G$ acts on $\E=W_r^{\oplus(d+1)}$ diagonally on the rows.
	This implies (1), and (3) follows directly.
	The complement of an invariant subspace with respect to an invariant scalar product is invariant, and a fixed vector of $\Eperp$ would lie in $\E^G\cap\Eperp=0$.
\end{proof}

\medbreak
Applying the previous lemma to our situation we get the following claim.

\medskip
\begin{lemma}
	\label{lem : fibres over the fixed sphere}
	Let $G\leq\sym_r$ be an elementary abelian group.
	\begin{compactenum}[\rm (1)]
	\item For every $y\in S(\E^G)$ the set $U_{a,y}$ is a $G$-box.
	\item If $\LL\subseteq\K$ is a $G$-invariant sub-complex, then over the fixed sphere $S(\E^G)$ the fibres of the first projection $\pr_1\colon T_{a,\LL}\longrightarrow S(\E)$ are the induced sub-complexes $\LL[U]$ for some $G$-boxes $U$.	
	\end{compactenum}
\end{lemma}
\begin{proof}
	(1) If $y\in\E^G$, then $f^y_i=f^y_{i'}$ whenever $i$ and $i'$ lie in the same $G$-orbit according to  Lemma~\ref{lem : decomposition}(1) and~\eqref{eq : affine functions}.
	By Lemma~\ref{lem : locality}, every $U_{a,y}\cap Q_j$ is a non-empty union of $G$-orbits.
	(2) The claim directly follows from the claim (1).
\end{proof}

\medskip
Thus, the $G$-boxes form a finite family which does not depend on the affine map $a$.

%----------------------------------------------------------------------------------------------
\medskip
\subsection{Borel construction and Euler classes}
\label{subsec : Borel}
%----------------------------------------------------------------------------------------------
Let $G$ be a finite group, and let $\FF$ be a field.
For a $G$-space $X$ consider the Borel construction $X_G:=\mathrm{E}G\times_GX$, and the Borel fibre bundle
\begin{equation}
	\label{eq : Borel}
	\xymatrix{
		X\ar[r] & \mathrm{E}G\times_G X \ar^-{\pi_X}[r] & \mathrm{B}G .
	}
\end{equation}
The (Borel) equivariant cohomology of $X$ is $H^*_G(X;\FF):=H^*(\mathrm{E}G\times_GX;\FF)$, and for an invariant subspace $X'\subseteq X$ we set $H^*_G(X,X';\FF):=H^*(X_G,X'_G;\FF)$; for background see~\cite[Ch.\,III]{Hsiang1975}, \cite[Sec.\,3.3]{BlagojevicZiegler2017}, and for the cohomology rings $H^*(\mathrm{B}G;\FF_p)$ of elementary abelian groups $G$ see~\cite{Adem2004}.
An equivariant map $f\colon X\longrightarrow Y$ induces the map $f_G:=\mathrm{id}\times_Gf\colon X_G\longrightarrow Y_G$, and additionally $\pi_X=\pi_Y\circ f_G$.
If $G$ acts trivially on $X$, then $X_G=\mathrm{B}G\times X$.
We take as a model of $\mathrm{E}G$ a contractible free $G$-CW-complex.
In particular, the Borel construction of a compact Hausdorff $G$-space $X$ is a paracompact Hausdorff space: the product $\mathrm{E}G\times X$ of a CW complex and a compact Hausdorff space is paracompact Hausdorff, and $X_G$ is its image under the quotient map by the finite group $G$, which is a closed map.

\medskip
A finite dimensional real representation $R$ of $G$ defines the vector bundle $\mathrm{E}G\times_GR\longrightarrow\mathrm{B}G$.
If this bundle is orientable over $\FF_p$, its mod $p$ Euler class is the {\em equivariant Euler class} $e_G(R)\in H^{\dim R}(\mathrm{B}G;\FF_p)$.
It is natural, multiplicative under direct sums, and it vanishes if $R$ has a non-zero fixed vector, consult for example \cite[Ch.\,III]{Hsiang1975}, \cite[\S9]{MilnorStasheff1974}.

\medskip
\begin{lemma}
	\label{lem : Euler class}
	Let $G$ be an elementary abelian group and $R$ a finite dimensional real representation of $G$ with $R^G=0$.
	Then the bundle $\mathrm{E}G\times_GR\longrightarrow\mathrm{B}G$ is orientable over $\FF_p$ and $e_G(R)\neq0$.
\end{lemma}
This is a well-known fact.
For $p$ odd the group $G$ has odd order, so every element of $G$ acts on $R$ with determinant $1$, and the bundle is orientable; for $p=2$ every vector bundle is orientable over $\FF_2$.
The non-vanishing of $e_G(R)$ is a consequence of the following corollary of the localisation theorem~\cite[Cor.\,1, p.\,45]{Hsiang1975}, which in the case of finite groups holds only for elementary abelian groups; we present it in the form stated in~\cite[Proof of Thm.\,3.11]{BlagojevicZiegler2017}.

\medskip
\begin{theorem*}
	Let $p$ be a prime, $G=(\Z/p)^\alpha$ with $\alpha\geq 1$, and let $X$ be a finite $G$-CW complex.
	The fixed point set $X^G$ of the space $X$ is non-empty if and only if the map in cohomology $H^*(\mathrm{B}G;\FF_p)\longrightarrow H^*(\mathrm{E}G\times_{G} X;\FF_p)$, induced by the projection $\mathrm{E}G\times_{G} X\longrightarrow\mathrm{B}G$, is a monomorphism.
\end{theorem*}

\medskip
Indeed, if $G$ is non-trivial, then $R^G=0$ means that $S(R)^G=\emptyset$, and so $\pi_{S(R)}^*\colon H^*(\mathrm{B}G;\FF_p)\longrightarrow H^*_G(S(R);\FF_p)$ is not a monomorphism.
By the Gysin sequence of the sphere bundle $S(R)_G\longrightarrow\mathrm{B}G$ the kernel of $\pi_{S(R)}^*$ is the ideal generated by $e_G(R)$~\cite[\S12]{MilnorStasheff1974}.
Hence $e_G(R)\neq0$.
If $G$ is trivial, then $R=R^G=0$ and $e_G(R)=1$.
For $G\cong(\Z/p)^\alpha$ acting regularly on $[r]$, $r=p^\alpha$, by the work of Mann and Milgram~\cite{Mann1982} the class $e_G(\E)=e_G(W_r^{\oplus(d+1)})$ is given as a product of non-zero linear forms in the polynomial part of $H^*(\mathrm{B}G;\FF_p)$; see also~\cite{Adem2004} and~\cite[Proof of Thm.\,3.11]{BlagojevicZiegler2017}.

\medskip
The next lemma is the equivariant version of the Vietoris--Begle mapping theorem; consult Theorem~\ref{lem : Vietoris-Begle}.

\medskip
\begin{lemma}
	\label{lem : fibrewise}
	Let $G$ be a finite group, $\FF$ a field, $\tau\geq0$, and let $f\colon X\longrightarrow Y$ be an equivariant continuous surjection of compact Hausdorff $G$-spaces such that $\widetilde H^i(f^{-1}(y);\FF)=0$ for all $y\in Y$ and $i\leq\tau$.
	Then $f_G^*\colon H^i_G(Y;\FF)\longrightarrow H^i_G(X;\FF)$ is an isomorphism for $i\leq\tau$ and a monomorphism for $i=\tau+1$.
	
	\smallskip\noindent
	In particular, if $X\neq\emptyset$ and $\widetilde H^i(X;\FF)=0$ for $i\leq\tau$, then $\pi_X^*\colon H^i(\mathrm{B}G;\FF)\longrightarrow H^i_G(X;\FF)$ is an isomorphism for $i\leq\tau$ and a monomorphism for $i=\tau+1$.
\end{lemma}
\begin{proof}
	Since $X$ is compact and $Y$ is Hausdorff, the map $f$ is proper, and so the map $\mathrm{id}\times f\colon\mathrm{E}G\times X\longrightarrow\mathrm{E}G\times Y$ is closed.
	The quotient map $\mathrm{E}G\times Y\longrightarrow Y_G$ is closed, because $G$ is finite, and the preimage in $\mathrm{E}G\times X$ of a closed subset of $X_G$ is closed; hence the induced map $f_G\colon X_G\longrightarrow Y_G$ is closed, and it is a surjection of paracompact Hausdorff spaces by Section~\ref{subsec : Borel}.
	Over $[z,y]$ the fibre of $f_G$ is $\{[z,x] : f(x)=y\}$, because $G$ acts freely on $\mathrm{E}G$, and $x\longmapsto[z,x]$ is a continuous bijection from the compact space $f^{-1}(y)$ onto this fibre, which is Hausdorff; so the fibre is homeomorphic to $f^{-1}(y)$.
	Now Theorem~\ref{lem : Vietoris-Begle} applied to $f_G$ implies the claim of the lemma. 
	For the last claim just take $Y=\pt$.
\end{proof}

%----------------------------------------------------------------------------------------------
\medskip
\subsection{The equivariant Poincar\'e dual of a fixed point}
\label{subsec : Thom class}
%----------------------------------------------------------------------------------------------
Let $G\leq\sym_r$ be an elementary abelian group of order $|G|=p^{\alpha}$ which acts non-trivially on $[r]$ and has at least two orbits, and let $m:=\dim\Eperp$.
Then by Lemma~\ref{lem : decomposition}, $\dim S(\E^G)=N-1-m$, and
\[
	1\ \leq\ \dim S(\E^G)\ \leq\ N-3,
\]
because there are at least two orbits, and because some orbit has at least $p\geq2$ elements so that $m\geq(d+1)(p-1)\geq 2$.
All cohomology in this subsection has coefficients in $\FF_p$, which we omit.
Let
\[
	e:=e_G(\Eperp)\in H^m(\mathrm{B}G),
\]
which is defined and non-zero by Lemma~\ref{lem : decomposition}(2) and Lemma~\ref{lem : Euler class}, and let 
\[
[S(\E^G)]^*\in H^{N-1-m}(S(\E^G))\cong\FF_p
\] 
be a fixed generator.
Since $G$ acts trivially on $S(\E^G)$, the K\"unneth formula gives that 
\begin{equation}
	\label{eq : Kunneth fixed sphere}
	H^{N-1}_G(S(\E^G))=H^{N-1}(\mathrm{B}G\times S(\E^G))\cong\big(H^{N-1}(\mathrm{B}G)\otimes1\big)\ \oplus\ \big(H^{m}(\mathrm{B}G)\otimes[S(\E^G)]^*\big).
\end{equation}
We call the second summand the {\em mixed} summand.

\medskip
\begin{lemma}[Equivariant Poincar\'e dual of a fixed point]
	\label{lem : Thom class}
	There is a class $\Theta\in H^{N-1}_G(S(\E))$ whose restriction to the fixed sphere is the following non-zero cohomology class:
	\[
		\Theta|_{S(\E^G)_G}\ =\ e\times[S(\E^G)]^*\ = \ e_G(\Eperp)\times[S(\E^G)]^*\in\ H^{N-1}_G(S(\E^G))=H^{N-1}(\mathrm{B}G\times S(\E^G)).
	\]
\end{lemma}

\medskip
In other words, the restriction of $\Theta$ has no component in the summand $H^{N-1}(\mathrm{B}G)\otimes1$, and its component in the mixed summand is $e\otimes[S(\E^G)]^*$.
In the usual non-equivariant setting, the restriction of any class of $H^{N-1}(S(\E))$ to the sphere $S(\E^G)$ of smaller dimension vanishes.

Equivariantly the situation is different, and there is a geometric reason for it, which explains the name of the class although the proof below does not use it.
Let $x_0\in S(\E^G)$ be a fixed point.
Non-equivariantly, the generator of $H^{N-1}(S(\E))$ is the Poincar\'e dual of the point $x_0$.
By the excess intersection formula, the restriction of this class to the submanifold $S(\E^G)$ through $x_0$ is the Poincar\'e dual of $x_0$ in $S(\E^G)$, that is $[S(\E^G)]^*$, multiplied by the Euler class of the normal bundle of $S(\E^G)$ in $S(\E)$.
This normal bundle is the trivial bundle $S(\E^G)\times\Eperp$, so its Euler class vanishes, and with it the restriction.
After the Borel construction the normal bundle becomes $(S(\E^G)\times\Eperp)_G\cong(\mathrm{E}G\times_G\Eperp)\times S(\E^G)$, a bundle over $\mathrm{B}G\times S(\E^G)$ which is no longer trivial: its Euler class is $e\times1$ with $e=e_G(\Eperp)\neq0$.
So the equivariant Poincar\'e dual of the fixed point $x_0$, which is the image of the generator of $H^{N-1}_G(S(\E),S(\E)-x_0)$ in $H^{N-1}_G(S(\E))$, restricts to $(e\times1)\cdot(1\times[S(\E^G)]^*)=e\times[S(\E^G)]^*$, so exactly what Lemma~\ref{lem : Thom class} proves. Moreover, any class $\Theta$ with the same restriction would also work, and the proof constructs one directly.

\medskip

\begin{proof}[Proof of Lemma~\ref{lem : Thom class}]
We determine the image of the restriction
\[
	\iota^*\colon H^{N-1}_G(S(\E))\longrightarrow H^{N-1}_G(S(\E^G)),
\]
which is induced by the inclusion $\iota\colon S(\E^G)\longrightarrow S(\E)$, using the exact sequence of the pair $(S(\E),S(\E^G))$,
\begin{equation}
	\label{eq : pair LES}
	  H^{N-1}_G(S(\E))\xrightarrow{\ \iota^*\ }H^{N-1}_G(S(\E^G))\xrightarrow{\ \delta\ }H^{N}_G(S(\E),S(\E^G)).
\end{equation}
Note that in this exact sequence all maps are $H^*(\mathrm{B}G)$-linear, up to a sign.

Now, for $\omega\in H^{N-1}(\mathrm{B}G)$ we have $\delta(\omega\times1)=0$, because $\omega\times1=\iota^*\pi_{S(\E)}^*\omega$ and $\delta\iota^*=0$.
Hence, by~\eqref{eq : Kunneth fixed sphere} and the $H^*(\mathrm{B}G)$-linearity of $\delta$, the kernel of $\delta$ in degree $N-1$ is determined by the classes $c\in H^m(\mathrm{B}G)$ with
\begin{equation}
	\label{eq : goal}
	\delta\big(c\times[S(\E^G)]^*\big)=\pm\,c\cdot\delta\big(1\times[S(\E^G)]^*\big)=0.
\end{equation}
We show that~\eqref{eq : goal} holds if and only if $c\in\FF_p\,e$.
For the lemma itself only the case $c=e$ is needed.

\medskip
We proceed in three steps and use the vector bundle $\xi:\mathrm{E}G\times_G\Eperp\longrightarrow\mathrm{B}G$ of rank $m$ with Euler class $e$, and its associated sphere bundle $S(\xi)=S(\Eperp)_G$ and projection $\pi_\xi\colon S(\xi)\longrightarrow\mathrm{B}G$.
In our case the sphere bundle $S(\xi)$ is connected, as $m\geq2$.

\medskip
	{\em Step 1: }
	Write $y=v+w$ with $v\in\E^G$ and $w\in\Eperp$.
	The equivariant homotopy $(v,w)\mapsto(v+(1-t)w)/\|v+(1-t)w\|$ deforms $S(\E)-S(\Eperp)$ onto $S(\E^G)$, so the inclusion of pairs $(S(\E),S(\E^G))\longrightarrow(S(\E),S(\E)-S(\Eperp))$ induces an isomorphism in equivariant cohomology.
	Next, the open invariant set $\mathcal W:=\{v+w\in S(\E) : \|v\|<\|w\|\}$ is equivariantly homeomorphic to $S(\Eperp)\times D(\E^G)$, where $D(\E^G)$ is the open unit disc, by $v+w\mapsto(w/\|w\|,v/\|w\|)$ with inverse $(w,x)\mapsto(w+x)/\sqrt{1+\|x\|^2}$.
	This homeomorphism maps $\mathcal W\cap S(\Eperp)$ onto $S(\Eperp)\times0$.
	Excising the closed set $S(\E)-\mathcal W$, which lies in the open set $S(\E)-S(\Eperp)$, and using that $G$ acts trivially on $\E^G$, we get
	\[
		H^*_G(S(\E),S(\E^G))\cong H^*_G(\mathcal W,\mathcal W-S(\Eperp))\cong H^*\big(S(\xi)\times(D(\E^G),D(\E^G)-0)\big).
	\]
	With a generator $u_0\in H^{N-m}(D(\E^G),D(\E^G)-0)\cong\FF_p$ the K\"unneth formula gives an $H^*(\mathrm{B}G)$-linear isomorphism
	\[
		\Phi\colon H^{j}(S(\xi))\longrightarrow H^{j+N-m}_G(S(\E),S(\E^G)),\qquad x\longmapsto x\times u_0 .
	\]

	\medskip
	{\em Step 2:}
	The class 
	\[
	\delta(1\times[S(\E^G)]^*)\ \in \ H^{N-m}_G(S(\E),S(\E^G))=\Phi(H^0(S(\xi)))\cong\FF_p
	\] 
	is not zero.
	Indeed, suppose that $\delta(1\times[S(\E^G)]^*)=0$.
	The exact sequence of the pair $(S(\E),S(\E^G))$ in degree $N-1-m$, the analogue of~\eqref{eq : pair LES}, gives a class $z\in H^{N-1-m}_G(S(\E))$ with $\iota^*z=1\times[S(\E^G)]^*$.
	Fix a point $b\in\mathrm{B}G$, and consider the inclusions of the fibres over $b$ of the Borel fibrations,
	\[
		j\colon S(\E)\longrightarrow S(\E)_G\qquad\text{and}\qquad j'\colon S(\E^G)\longrightarrow S(\E^G)_G=\mathrm{B}G\times S(\E^G),\quad x\longmapsto(b,x).
	\]
	They satisfy $\iota_G\circ j'=j\circ\iota$, where on the left $\iota_G\colon S(\E^G)_G\longrightarrow S(\E)_G$ is the map induced by $\iota$.
	Hence
	\[
		[S(\E^G)]^*=j'^*\big(1\times[S(\E^G)]^*\big)=j'^*\iota^*z=\iota^*j^*z,
	\]
	where on the right $\iota^*\colon H^{N-1-m}(S(\E))\longrightarrow H^{N-1-m}(S(\E^G))$ is the non-equivariant restriction.
	But $H^{N-1-m}(S(\E))=0$, because $S(\E)$ is a sphere of dimension $N-1$ and $1\leq N-1-m\leq N-2$.
	So $[S(\E^G)]^*=0$, which contradicts the choice of $[S(\E^G)]^*$ as a generator.
	Hence $\delta(1\times[S(\E^G)]^*)$ is a non-zero multiple of $\Phi(1)$.

	\medbreak
	Let $c\in H^m(\mathrm{B}G)$.
	The ring $H^*(\mathrm{B}G)$ acts on $H^*(S(\xi))$ through $\pi_\xi^*$, so $c\cdot1=\pi_\xi^*c$ in $H^m(S(\xi))$.
	Since $\delta$ is $H^*(\mathrm{B}G)$-linear up to a sign and $\Phi$ is $H^*(\mathrm{B}G)$-linear, the class
	\[
		\delta\big(c\times[S(\E^G)]^*\big)=\pm\,c\cdot\delta\big(1\times[S(\E^G)]^*\big)
	\]
	is a non-zero multiple of $c\cdot\Phi(1)=\Phi(c\cdot1)=\Phi(\pi_\xi^*c)$.

	\medskip
	{\em Step 3:}
	For the $\FF_p$-oriented $(m-1)$-sphere bundle $S(\xi)\longrightarrow\mathrm{B}G$ the Gysin sequence~\cite[\S12]{MilnorStasheff1974}, \cite[\S4.D]{Hatcher2002}
	\[
		\cdots\longrightarrow H^{j-m}(\mathrm{B}G)\xrightarrow{\ \cdot e\ }H^j(\mathrm{B}G)\xrightarrow{\ \pi_\xi^*\ }H^j(S(\xi))\longrightarrow H^{j-m+1}(\mathrm{B}G)\longrightarrow\cdots
	\]
	is exact.
	For $j=m$ it shows that the kernel of $\pi_\xi^*\colon H^m(\mathrm{B}G)\longrightarrow H^m(S(\xi))$ is $e\cdot H^0(\mathrm{B}G)=\FF_p\,e$.
	Now let $c\in H^m(\mathrm{B}G)$.
	By Step~2 the class $\delta(c\times[S(\E^G)]^*)$ is a non-zero multiple of $\Phi(\pi_\xi^*c)$, and $\Phi$ is injective, since it is an isomorphism by Step~1.
	Therefore $\delta(c\times[S(\E^G)]^*)=0$ if and only if $\pi_\xi^*c=0$, that is, if and only if $c\in\FF_p\,e$.
	This proves the desired claim~\eqref{eq : goal}.
	
	\medbreak
	In particular, $\delta(e\times[S(\E^G)]^*)=0$, and by the exactness of~\eqref{eq : pair LES} there is $\Theta\in H^{N-1}_G(S(\E))$ with $\iota^*\Theta=e\times[S(\E^G)]^*$.
\end{proof}

\medskip
The proof gives more than the existence of $\Theta$.
Indeed, the image of $\iota^*$ in degree $N-1$ is the kernel of $\delta$, which is
\[
	\big(H^{N-1}(\mathrm{B}G)\otimes1\big)\ \oplus\ \FF_p\,\big(e\otimes[S(\E^G)]^*\big).
\]
In particular, the classes in the mixed summand which extend to $S(\E)$ are exactly the multiples of $e\times[S(\E^G)]^*$.

%----------------------------------------------------------------------------------------------
\medskip
\subsection{The $G$-box  criterion}
\label{sec : fixed sphere theorem}
%----------------------------------------------------------------------------------------------
In this subsection we state and prove the $G$-box criterion, which is the central topological result of the paper.

\medskip
Recall from Section~\ref{subsec : group actions} that, identifying each column $Q_j$ with the set of labels $[r]$ via $p_{j,i}\mapsto i$, a set $U\subseteq V$ is a $G$-box if every $U\cap Q_j$ is a non-empty union of $G$-orbits.
So a $G$-box is obtained by choosing, independently for every column, a non-empty union of orbits, and if $G$ has $t$ orbits on $[r]$, then there are $(2^t-1)^{n+1}$ different $G$-boxes.
The largest one is $V$, and the minimal ones have exactly one orbit in every column.
Only the induced sub-complexes $\LL[U]$ of these sets $U$ enter the assumption of the next theorem.
Figure~\ref{fig : G-boxes} shows several $G$-boxes for $r=6$.

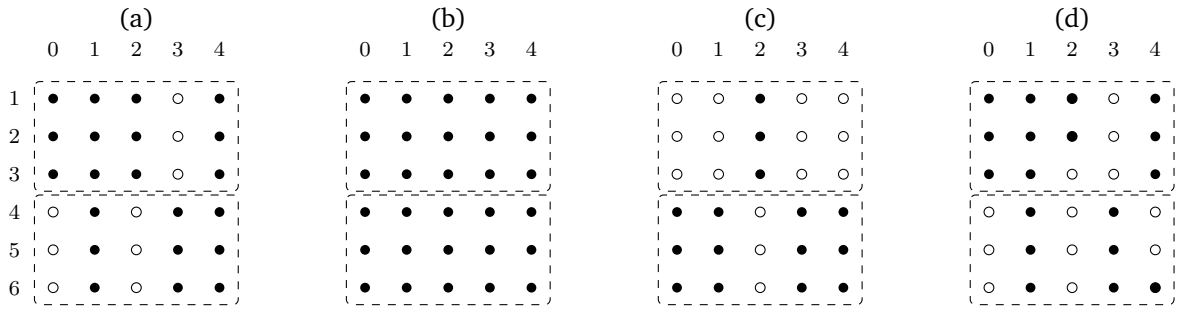
\begin{figure}[ht]
\centering
\begin{tikzpicture}[x=5.5mm,y=5mm,
	pt/.style={circle,fill=black,inner sep=1.3pt},
	hollow/.style={circle,draw,fill=white,inner sep=1.3pt},
	lab/.style={font=\scriptsize}]
	\foreach \X/\T/\A/\B in {0/{(a)}/{1,1,1,0,1}/{0,1,0,1,1},
	                        7.5/{(b)}/{1,1,1,1,1}/{1,1,1,1,1},
	                        15/{(c)}/{0,0,1,0,0}/{1,1,0,1,1},
	                        22.5/{(d)}/{1,1,0,0,1}/{0,1,0,1,0}}{
		\begin{scope}[xshift=\X*5.5mm]
			\node[font=\small] at (2,1.1) {\T};
			\foreach \j in {0,...,4}{\node[lab] at (\j,0.3) {$\j$};}
			\foreach \a/\b in {1/3,4/6}{\draw[dashed,rounded corners=2pt] (-0.45,-\a+0.45) rectangle (4.45,-\b-0.45);}
			\foreach \m [count=\jj from 0] in \A {
				\foreach \i in {1,2,3}{\ifnum\m=1 \node[pt] at (\jj,-\i) {}; \else \node[hollow] at (\jj,-\i) {}; \fi}}
			\foreach \m [count=\jj from 0] in \B {
				\foreach \i in {4,5,6}{\ifnum\m=1 \node[pt] at (\jj,-\i) {}; \else \node[hollow] at (\jj,-\i) {}; \fi}}
			\ifdim\X pt=0pt
				\foreach \i in {1,...,6}{\node[lab,anchor=east] at (-0.55,-\i) {$\i$};}
			\fi
		\end{scope}}
	\begin{scope}[xshift=22.5*5.5mm]
		\node[pt] at (2,-1) {}; \node[pt] at (2,-2) {};
		\node[pt] at (4,-6) {};
	\end{scope}
\end{tikzpicture}
\caption{\small Subsets of the vertex set $V$ for $r=6$, $n=4$, and the group $G\cong\Z/3$ generated by the permutation $(1\,2\,3)(4\,5\,6)$ of $[6]$, whose orbits $\{1,2,3\}$ and $\{4,5,6\}$ are framed; filled vertices belong to the set.
(a), (b), (c) are $G$-boxes: in every column the set is $\{1,2,3\}$, $\{4,5,6\}$ or $\{1,\dots,6\}$, that is, a non-empty union of orbits; (b) is the $G$-box $V$.
There are $3^5=243$ $G$-boxes here, and the assumption of Theorem~\ref{th : fixed sphere} asks that $\LL[U]$ be homologically $(N-2)$-connected for each of them, and $(N-1)$-connected for $U=V$.
(d) is not a $G$-box: in column $2$ it contains only part of the orbit $\{1,2,3\}$, and in column $4$ it contains part of the orbit $\{4,5,6\}$.
The induced sub-complex $\LL[U]$ of a set $U$ which is not a $G$-box plays no role in the theorem.}
\label{fig : G-boxes}
\end{figure}

\medskip
\begin{theorem}[$G$-box criterion]
	\label{th : fixed sphere}
	Let $d\geq 1$ and $r\geq 2$ be integers, and let $N:=(d+1)(r-1)$.
	Further, take $p$ to be a prime and  $G\leq\sym_r$ an elementary abelian group isomorphic to $(\Z/p)^\alpha$ for some $\alpha\geq0$. 
	Suppose $n\geq0$ is an integer and  $a\colon\Delta_n\longrightarrow\R^d$ is an arbitrary affine map.
	
	\smallskip\noindent
	If $\LL\subseteq(\Delta_n)^{*r}_{\Delta(2)}$ is a $G$-invariant sub-complex which is homologically $(N-1)$-connected over $\FF_p$, and such that $\LL[U]$ is homologically $(N-2)$-connected over $\FF_p$ for every $G$-box $U$, then $\LL$ contains a Tverberg face of $a$, that is, there are pairwise disjoint faces $F_1,\dots,F_r$ of $\Delta_n$ with $F_1*\cdots*F_r\in\LL$ and $a(F_1)\cap\cdots\cap a(F_r)\neq\emptyset$.
\end{theorem}

\medskip
The two connectivity assumptions enter the proof at different places: the $(N-1)$-connectivity of $\LL$ itself, that is, of $\LL[V]$, is used to compare the incidence space with the classifying space, and the $(N-2)$-connectivity of the complexes $\LL[U]$ is used for the fibres over the fixed sphere.
If $G$ has one orbit, then $r=p^\alpha$, $G$ acts regularly on $[r]$, the only $G$-box is $V$, and the assumption says only that $\LL$ is homologically $(N-1)$-connected.
In this case Theorem~\ref{th : fixed sphere} follows from the classical work of \"{O}zaydin~\cite{Oezaydin1987} and Volovikov~\cite{Volovikov1996}; see also~\cite[Thm.\,3.11 and Rem.\,3.12]{BlagojevicZiegler2017}.
For the sake of completeness we include a proof.

\begin{proof}[Proof of Theorem~\ref{th : fixed sphere} when $G$ has one orbit]
	As before, cohomology has coefficients in the field $\FF_p$.
	Suppose that $\LL$ contains no Tverberg face of $a$.
	By Lemma~\ref{lem : CSTM} the test map $A$ has no zero on $\LL$, and $\nu:=A/\|A\|\colon\LL\longrightarrow S(\E)$ is an equivariant map.
	Since $G$ has one orbit, the fixed subspace $\E^G$ is $0$ according to Lemma~\ref{lem : decomposition}, and therefore $e_G(\E)\neq0$ by Lemma~\ref{lem : Euler class}.
	In the Gysin sequence of the sphere bundle $S(\E)_G\longrightarrow\mathrm{B}G$ the Euler class $e_G(\E)$ lies in the kernel of $\pi_{S(\E)}^*$~\cite[\S12]{MilnorStasheff1974}; in the language of the Serre spectral sequence this is the transgression used in~\cite[Proof of Thm.\,3.11]{BlagojevicZiegler2017}.
	Hence 
	\[
	\pi_\LL^*e_G(\E)=\nu_G^*\pi_{S(\E)}^*e_G(\E)=0
	\] 
	in $H^N_G(\LL)$.
	But $\LL$ is homologically $(N-1)$-connected, so $\pi_\LL^*$ is injective in degree $N$ by Lemma~\ref{lem : fibrewise}, a contradiction.
\end{proof}

\medskip
This part of the proof is the argument of~\cite[Proof of Thm.\,3.11]{BlagojevicZiegler2017}, written with the Gysin sequence, and it uses only the equivariant map $\nu$.
For a continuous map $\Delta_n\longrightarrow\R^d$ the test map is defined by the same formula~\eqref{eq : test map}, and the first part of the proof of Lemma~\ref{lem : CSTM} shows that it has a zero on $\LL$ if and only if $\LL$ contains a Tverberg face.

\medbreak
Hence Theorem~\ref{th : fixed sphere} holds for continuous maps $\Delta_n\longrightarrow\R^d$ when $G$ has one orbit, and so does every consequence which uses only this case.

\begin{proof}[Proof of Theorem~\ref{th : fixed sphere} when $G$ has at least two orbits]
	Again, all cohomology has coefficients in $\FF_p$.
	Suppose that $\LL$ contains no Tverberg face of $a$, and let
	\[
		T:=T_{a,\LL}=\{(y,\lambda)\in S(\E)\times\LL : \supp(\lambda)\subseteq U_{a,y}\}
	\]
	be the incidence space of $\LL$ with respect to the affine map $a\colon \Delta_n\longrightarrow\R^d$, see Definition~\ref{def : incidence space}, with the projections $\pr_1\colon T\longrightarrow S(\E)$ and $\pr_2\colon T\longrightarrow\LL$.
	Since $\LL$ is $G$-invariant, $T$ is a compact $G$-space and both projections are equivariant, see Section~\ref{subsec : group actions}.
	For a $G$-space $X$ we write $\pi_X\colon X_G\longrightarrow\mathrm{B}G$ for the projection.
	We proceed in six steps.

	\medskip
	{\em Step 1:}
	It follows from Proposition~\ref{prop : projection to L} that $\pr_2$ is an equivariant surjection of compact spaces with contractible fibres, and so the homomorphism 
		\[
		(\pr_2)_G^*\colon H^*_G(\LL)\longrightarrow H^*_G(T)
		\] 
		is an isomorphism according to Lemma~\ref{lem : fibrewise}.

	\medskip
	{\em Step 2:}
		According to the homological $(N-1)$-connectivity of $\LL$ and Lemma~\ref{lem : fibrewise}, the induced map in cohomology $\pi_\LL^*\colon H^{N-1}(\mathrm{B}G)\longrightarrow H^{N-1}_G(\LL)$ is an isomorphism.
		Since $\pi_T=\pi_\LL\circ(\pr_2)_G$, Step~1 implies that also
		\[
		\pi_T^*\colon H^{N-1}(\mathrm{B}G)\longrightarrow H^{N-1}_G(T)
		\] 
		is an isomorphism.

	\medskip
	{\em Step 3:}
		If $G$ acts non-trivially on $[r]$, let $\Theta\in H^{N-1}_G(S(\E))$ and $e\in H^m(\mathrm{B}G)$ be as in Lemma~\ref{lem : Thom class}.
		This is the situation of Section~\ref{subsec : Thom class}, whose standing assumptions, and hence Lemma~\ref{lem : Thom class}, require a non-trivial action and at least two orbits.
		If $G$ acts trivially, then $S(\E^G)=S(\E)$, and we set $\Theta:=1\times[S(\E)]^*\in H^{N-1}(\mathrm{B}G\times S(\E))$, $e:=1$ and $m:=0$.
		Then $\Theta|_{S(\E^G)_G}=\Theta=e\times[S(\E^G)]^*$, so the conclusion of Lemma~\ref{lem : Thom class} holds in this case as well.
		In both cases Step~2 gives a class $\omega\in H^{N-1}(\mathrm{B}G)$ with
		\begin{equation}\label{eq:identity_step_3}
			(\pr_1)_G^*\Theta=\pi_T^*\omega \ \in \  H^{N-1}_G(T).
		\end{equation}

	\medskip
	{\em Step 4:}
		Let $T':=\pr_1^{-1}(S(\E^G))$, and let $\rho\colon T'\longrightarrow S(\E^G)$ be the restriction of $\pr_1$.
		The set $T'$ is closed in $T$ and therefore compact.
		Additionally, it is also $G$-invariant.
		In summary, $T'$ is a compact $G$-space and $\rho$ is an equivariant map.
		
		\smallskip
		Next we restrict the identity \eqref{eq:identity_step_3} to the subspace $T'_G\subseteq T_G$.
		Since $\pr_1$ maps $T'$ into $S(\E^G)$, its restriction to $T'$ is $\iota\circ\rho$, where $\iota\colon S(\E^G)\longrightarrow S(\E)$ is the inclusion, and so the restriction of $(\pr_1)_G$ to $T'_G$ is $\iota_G\circ\rho_G$.
		Hence the left side of the identity restricts to
		\[
			\big((\pr_1)_G^*\Theta\big)\big|_{T'_G}=\rho_G^*\iota_G^*\Theta=\rho_G^*\big(\Theta|_{S(\E^G)_G}\big)=\rho_G^*\big(e\times[S(\E^G)]^*\big),
		\]
		where the last equality is Lemma~\ref{lem : Thom class}, or the choice of $\Theta$ in the case of the trivial action.

		On the right side, the restriction of $\pi_T$ to $T'_G$ is $\pi_{T'}=\pi_{S(\E^G)}\circ\rho_G$.
		Since $G$ acts trivially on $S(\E^G)$, we have $S(\E^G)_G=\mathrm{B}G\times S(\E^G)$, the map $\pi_{S(\E^G)}$ is the projection onto the first factor, and so $\pi_{S(\E^G)}^*\omega=\omega\times1$.
		Therefore the right side restricts to
		\[
			\big(\pi_T^*\omega\big)\big|_{T'_G}=\pi_{T'}^*\omega=\rho_G^*\pi_{S(\E^G)}^*\omega=\rho_G^*\big(\omega\times1\big).
		\]
		Comparing both restrictions we obtain
		\begin{equation}\label{eq:identity_step_3_2}
			\rho_G^*\big(e\times[S(\E^G)]^*\big)=\rho_G^*\big(\omega\times1\big)\ \in \ H^{N-1}_G(T').
		\end{equation}

	\medskip
	{\em Step 5:}
		The fibres of $\rho$ are the complexes $\LL[U_{a,y}]$ for $y\in S(\E^G)$.
		According to Lemma~\ref{lem : fibres over the fixed sphere} every $U_{a,y}$ is a $G$-box, so by the assumption of the theorem the fibres are homologically $(N-2)$-connected.
		In particular the fibres are non-empty, so $\rho$ is surjective.
		As $T'$ and $S(\E^G)$ are compact, Lemma~\ref{lem : fibrewise} with $\tau=N-2$ applies to the equivariant map $\rho$ and implies that $\rho_G^*$ is injective in degree $N-1$.
		Thus, according to \eqref{eq:identity_step_3_2} it follows that
		\begin{equation}\label{eq:omega}
			\omega\times1=e\times[S(\E^G)]^*\ \in \ H^{N-1}(\mathrm{B}G\times S(\E^G)).
		\end{equation}

	\medskip
	{\em Step 6:}
		From Lemma~\ref{lem : decomposition} we have that $\dim S(\E^G)=(d+1)(|[r]/G|-1)-1\geq d\geq1$.
		Hence the class $[S(\E^G)]^*$ has positive degree $N-1-m$, and the K\"unneth decomposition~\eqref{eq : Kunneth fixed sphere},
\[			 H^{N-1}(\mathrm{B}G\times S(\E^G))=\big(H^{N-1}(\mathrm{B}G)\otimes H^0(S(\E^G))\big)\ \oplus\ \big(H^{m}(\mathrm{B}G)\otimes H^{N-1-m}(S(\E^G))\big),
\]
		is a direct sum of two different summands.
		The left side of~\eqref{eq:omega} lies in the first summand, while the right side lies in the second, the mixed summand.
		Comparing the components of both sides of~\eqref{eq:omega} in this direct sum, we see that both sides have to vanish: $\omega\times1=0$ and $e\times[S(\E^G)]^*=0$.
		The first equality only says that $\omega=0$, which is not a contradiction by itself, but the second one is.
		Indeed, since $H^{N-1-m}(S(\E^G))\cong\FF_p$ is generated by $[S(\E^G)]^*$, the map $H^m(\mathrm{B}G)\longrightarrow H^m(\mathrm{B}G)\otimes H^{N-1-m}(S(\E^G))$, $c\mapsto c\otimes[S(\E^G)]^*$, is an isomorphism, and so $e=0$.
		However, $e=e_G(\Eperp)\neq0$ by Lemma~\ref{lem : Euler class} when $G$ acts non-trivially, and $e=1$ when $G$ acts trivially.
\end{proof}

\medskip
For the trivial group the proof becomes a non-equivariant argument.
There the class $\Theta=[S(\E)]^*$ lives on the whole sphere, and Step~5 needs the connectivity of $\LL[U_{a,y}]$ for all $y\in S(\E)$.
This is a condition of the same kind as the one obtained from the covering scheme at the end of Section~\ref{subsec : covering scheme}.

%----------------------------------------------------------------------------------------------
%----------------------------------------------------------------------------------------------
\bigskip
\section{Proofs of the main results}
\label{sec : proofs}
\bigskip
%----------------------------------------------------------------------------------------------
%----------------------------------------------------------------------------------------------

In this section we prove the results stated in Section~\ref{sec : introduction}.
First we introduce the difference complexes and verify the assumption of the $G$-box criterion, Theorem~\ref{th : fixed sphere}, for them, which yields Theorem~\ref{th : difference} and Corollary~\ref{cor : matching}.
Then, with translated difference complexes, we double count Tverberg faces and prove Theorem~\ref{th : main} and Corollaries~\ref{cor : every r} and~\ref{cor : rate}, and, averaging over matchings, Theorem~\ref{th : averaged} and Corollary~\ref{cor : averaged rate}.

\medskip
Throughout this section $r=kq+s$, where $q=p^{\alpha}$ is a power of a prime $p$, so $\alpha\geq1$, $k\geq1$ and $s\geq0$.
Observe that in our setup the choice of $q=p^{\alpha}$ and respectively $s$ is not unique.
The set of labels $[r]$ is identified with

\[
	\Lambda=(\FF_p^{\,\alpha}\times[k])\sqcup\{\infty_1,\dots,\infty_s\}
\]
as in Section~\ref{sec : introduction}.
We let the group $G:=\FF_p^{\,\alpha}$, that is, the additive group of the vector space $\FF_p^{\,\alpha}$, which is isomorphic to $(\Z/p)^\alpha$, act on $\Lambda$ by translation of the $\FF_p^{\,\alpha}$-label:
\[
	g\cdot(u,\iota):=(u+g,\iota)
	\qquad\text{and}\qquad 
	g\cdot\infty_\ell:=\infty_\ell, \qquad 1\leq \ell\leq s.
\]
In this way $G$ becomes an elementary abelian subgroup of the symmetric group $\sym_r$ whose orbits are the $k$ blocks $\FF_p^{\,\alpha}\times\{\iota\}$, of size $q$, and the $s$ fixed labels $\infty_1,\dots,\infty_s$.
So a set $U\subseteq V$ is a $G$-box if and only if every $U\cap Q_j$ is a non-empty union of blocks and fixed labels.

%----------------------------------------------------------------------------------------------
\medskip
\subsection{Difference complexes}
\label{sec : difference complexes}
%----------------------------------------------------------------------------------------------
Difference complexes, as we will explain, are sub-complexes of the deleted join $\K$ given by constraints on the labels of the two vertices of each pair of an oriented matching.

\medskip
For a subset $\Gamma\subseteq\FF_p^{\,\alpha}$ let
\begin{multline*}
	R_\Gamma:=\big\{\big((u,\iota),(u+g,\iota')\big) : u\in\FF_p^{\,\alpha},\ g\in\Gamma,\ \iota,\iota'\in[k]\big\}\\
	\cup\ \big\{(c,c')\in\Lambda\times\Lambda : \{c,c'\}\cap\{\infty_1,\dots,\infty_s\}\neq\emptyset\big\}.
\end{multline*}
So $(c,c')\in R_\Gamma$ means that, if $c$ and $c'$ both lie in blocks, then the $\FF_p^{\,\alpha}$-label of $c'$ minus the $\FF_p^{\,\alpha}$-label of $c$ lies in $\Gamma$.
Furthermore, $(c,c')\in R_\Gamma$ implies $(g\cdot c,g\cdot c')\in R_\Gamma$ for every $g\in G$, and a pair $((u,\iota),(u,\iota'))$ with equal $\FF_p^{\,\alpha}$-labels lies in $R_\Gamma$ if and only if $0\in\Gamma$.
The relation $R_\Gamma$ is in general not symmetric: $(c,c')\in R_\Gamma$ if and only if $(c',c)\in R_{-\Gamma}$, where $-\Gamma=\{-g : g\in\Gamma\}$, and so $R_\Gamma$ is symmetric if and only if $\Gamma=-\Gamma$.

\medskip
For sets $S,S'\subseteq\Lambda$ we consider the bipartite graph $R_\Gamma[S,S']$ whose vertex set is the disjoint union $S\sqcup S'$ of a copy of $S$ and a copy of $S'$, and in which the copy of $c\in S$ is adjacent to the copy of $c'\in S'$ if $(c,c')\in R_\Gamma$.
The sets $S$ and $S'$ are not required to be disjoint. For an illustration see Figure~\ref{fig : difference graph}.

\begin{figure}[ht]
\centering
\begin{tikzpicture}[x=13mm,y=8mm,
	pt/.style={circle,fill=black,inner sep=1.4pt},
	lab/.style={font=\scriptsize}]
	\foreach \X/\G/\E in {0/{\{0,1\}}/1,6.2/{\{0\}}/0}{
		\begin{scope}[xshift=\X cm]
			\foreach \u in {0,1,2}{
				\node[pt] (a\u) at (0,-\u) {};
				\node[pt] (b\u) at (2,-\u) {};
				\node[lab,anchor=east] at (-0.1,-\u) {$(\u,1)$};
				\node[lab,anchor=west] at (2.1,-\u) {$(\u,2)$};}
			\foreach \u in {0,1,2}{\draw (a\u) -- (b\u);}
			\ifnum\E=1
				\draw (a0) -- (b1); \draw (a1) -- (b2); \draw (a2) -- (b0);
			\fi
			\node[lab] at (0,0.7) {$S$}; \node[lab] at (2,0.7) {$S'$};
			\node[font=\small] at (1,-2.8) {$\Gamma=\G$};
		\end{scope}}
\end{tikzpicture}
\caption{\small The graph $R_\Gamma[S,S']$ for $p=3$, $\alpha=1$, $k\geq2$, $S=\FF_3\times\{1\}$ and $S'=\FF_3\times\{2\}$. A vertex $(u,1)$ is adjacent to $(u+g,2)$ for $g\in\Gamma$.
If $(v,v')=(p_j,p_{j'})\in M$, $\Gamma_{v,v'}=\Gamma$ and $U$ is a $G$-box with $U\cap Q_j=S$ and $U\cap Q_{j'}=S'$, then this graph is a join factor of $\LL[U]$.
For $\Gamma=\{0,1\}$ (left) the differences $\pm1$ generate $\FF_3$, and the graph is a cycle of length six, which is connected.
In the case $\Gamma=\{0\}$ (right) the only difference is $0$, and the graph is a perfect matching with three components.}
\label{fig : difference graph}
\end{figure}
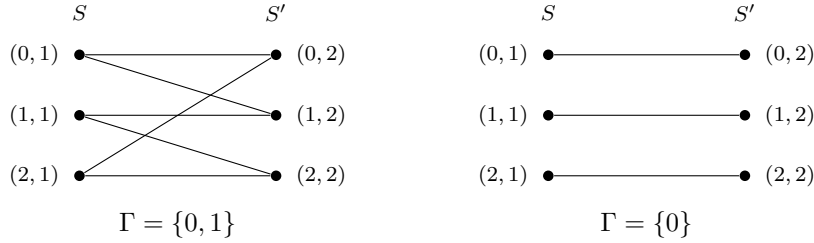

\medskip
The constraints we are imposing are on the pairs of an oriented matching.
Recall that an {\em oriented matching} $M$ on the vertex set of $\Delta_n$ is a set of ordered pairs $(v,v')$ of distinct vertices of $\Delta_n$ such that every vertex belongs to at most one pair.
In other words, $M$ is a matching together with an orientation of each of its edges.
Thus different pairs of $M$ are disjoint when considered as sets; in particular, if $(v,v')\in M$, then $(v',v)\notin M$.

\medskip
\begin{definition}
	\label{def : difference complex}
	Let $M$ be an oriented matching on the vertex set of $\Delta_n$, and suppose that to every pair $(v,v')\in M$ a subset $\Gamma_{v,v'}\subseteq\FF_p^{\,\alpha}$ is assigned, the set of {\em allowed differences} of this pair.
	The {\em difference complex} $\LL_M(\Gamma_{v,v'} : (v,v')\in M)$ is the sub-complex of $\K=(\Delta_n)^{*r}_{\Delta(2)}$ consisting of all faces $J$ with the following property: for every pair $(v,v')=(p_j,p_{j'})\in M$, if $p_{j,c}\in J$ and $p_{j',c'}\in J$ for some labels $c,c'\in\Lambda$, then $(c,c')\in R_{\Gamma_{v,v'}}$.
	If $\Gamma_{v,v'}=\Gamma$ for all pairs, we write $\LL_M(\Gamma)$.
\end{definition}

\medskip
Since different pairs of $M$ are disjoint, the requirement placed on the pair $(v,v')=(p_j,p_{j'})\in M$ in Definition~\ref{def : difference complex} only involves the two columns $Q_j$ and $Q_{j'}$, so different pairs constrain disjoint sets of columns.
Consequently, $\LL_M(\Gamma_{v,v'} : (v,v')\in M)$ is the join of the bipartite graphs $R_{\Gamma_{v,v'}}[\Lambda,\Lambda]$, viewed as $1$-dimensional complexes on $Q_j\sqcup Q_{j'}$ for $(v,v')=(p_j,p_{j'})\in M$, and of the $0$-dimensional complexes $Q_j$ for the vertices $p_j$ not covered by $M$; see Figure~\ref{fig : difference complex} for an example.
A facet of $\K$ lies in $\LL_M(\Gamma_{v,v'} : (v,v')\in M)$ if and only if it is a labelling of all vertices of $\Delta_n$ with the following property: if $v$ has the label $(u,\iota)$ and $v'$ has the label $(u',\iota')$ for a pair $(v,v')\in M$, then $u'-u\in\Gamma_{v,v'}$.
So the set $\Gamma_{v,v'}$ constrains only the labels of the two vertices $v$ and $v'$, and only when both labels lie in blocks; the vertices not covered by $M$ are not constrained at all.
The sets $\Gamma_{v,v'}$ for different pairs are independent of each other, and the order within a pair matters.
If a pair $(v,v')\in M$ is replaced by the reversed pair $(v',v)$, which gives another oriented matching, then the same difference complex is obtained when the set $\Gamma_{v,v'}$ is replaced by $-\Gamma_{v,v'}$.

\begin{figure}[ht]
\centering
\begin{tikzpicture}[x=11mm,y=7.5mm,
	pt/.style={circle,fill=black,inner sep=1.3pt},
	lab/.style={font=\scriptsize},
	blk/.style={thin},
	inf/.style={thin,black!45},
	fac/.style={line width=1.8pt}]
	\newcommand{\diffgraph}[6]{% #1 xshift, #2 list of differences, #3 left col name, #4 right col name, #5 facet left row, #6 facet right row
		\begin{scope}[xshift=#1 cm]
			\foreach \u/\y in {0/0,1/-1,2/-2}{
				\foreach \g in {#2}{
					\pgfmathtruncatemacro{\w}{mod(\u+\g,3)}
					\draw[blk] (0,\y) -- (1.6,-\w);}}
			\foreach \y in {0,-1,-2,-3}{\draw[inf] (0,-3) -- (1.6,\y); \draw[inf] (0,\y) -- (1.6,-3);}
			\draw[fac] (0,#5) -- (1.6,#6);
			\foreach \y/\l in {0/0,-1/1,-2/2,-3/\infty}{
				\node[pt] at (0,\y) {}; \node[pt] at (1.6,\y) {};
				\node[lab,anchor=east] at (-0.12,\y) {$\l$};
				\node[lab,anchor=west] at (1.72,\y) {$\l$};}
			\node[lab] at (0,0.75) {#3}; \node[lab] at (1.6,0.75) {#4};
		\end{scope}}
	\diffgraph{0}{0,1}{$Q_0$}{$Q_1$}{-1}{-2}
	\node at (2.85,-1.5) {$*$};
	\diffgraph{4.0}{1,2}{$Q_2$}{$Q_3$}{0}{-2}
	\node at (6.85,-1.5) {$*$};
	\begin{scope}[xshift=8cm]
		\foreach \y/\l in {0/0,-1/1,-2/2,-3/\infty}{\node[pt] at (0,\y) {}; \node[lab,anchor=west] at (0.12,\y) {$\l$};}
		\node[circle,draw,line width=1.2pt,inner sep=2.6pt] at (0,-3) {};
		\node[lab] at (0,0.75) {$Q_4$};
	\end{scope}
	\node[font=\small] at (0.8,-3.85) {$\Gamma_{p_0,p_1}=\{0,1\}$};
	\node[font=\small] at (4.8,-3.85) {$\Gamma_{p_2,p_3}=\{1,2\}$};
\end{tikzpicture}
\caption{\small The difference complex $\LL=\LL_M(\Gamma_{v,v'} : (v,v')\in M)$ for $r=4=1\cdot3+1$, so $p=3$, $\alpha=1$, $k=1$, $s=1$ and $\Lambda=\FF_3\sqcup\{\infty\}$, without block indices, and for $n=4$, the oriented matching $M=\{(p_0,p_1),(p_2,p_3)\}$, $\Gamma_{p_0,p_1}=\{0,1\}$ and $\Gamma_{p_2,p_3}=\{1,2\}$.
The complex is the join $\LL=R_{\{0,1\}}[Q_0,Q_1]*R_{\{1,2\}}[Q_2,Q_3]*Q_4$ of two bipartite graphs and the $0$-dimensional complex $Q_4$, since the vertex $p_4$ is not covered by $M$.
In the left graph a label $u$ of $p_0$ is joined to the labels $u$ and $u+1$ of $p_1$, and in the middle graph a label $u$ of $p_2$ is joined to the labels $u+1$ and $u+2$ of $p_3$; the grey edges join the fixed label $\infty$ to every label on the other side.
A facet of $\LL$ consists of one edge from each graph and one vertex of $Q_4$. As an example, the highlighted facet $p_0,\dots,p_4$ is represented by $1,2,0,2,\infty$, and indeed $2-1\in\{0,1\}$ and $2-0\in\{1,2\}$.
Each graph is connected, and $\LL$ is a join of top-concentrated complexes of dimension $4=n$, as in Proposition~\ref{prop : difference complex}.}
\label{fig : difference complex}
\end{figure}
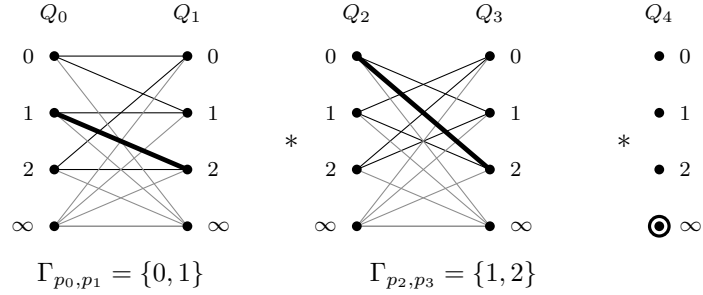
Two choices are used in this paper.
For the constant choice $\Gamma_{v,v'}=\Gamma$ the facets of $\LL_M(\Gamma)$ are exactly the labellings in Theorem~\ref{th : difference}, and the translates $\Gamma_{v,v'}=\Gamma+h_{v,v'}$ appear in the counting arguments of Section~\ref{sec : counting}.

\medskip
We say that a non-empty complex $K$ is {\em top-concentrated} if it is homologically $(\dim K-1)$-connected.
By the K\"unneth formula for joins~\eqref{eq : Kunneth for joins} and induction, a join $K_1*\cdots*K_l$ of top-concentrated complexes is top-concentrated of dimension $\sum_i(\dim K_i+1)-1$.
A non-empty $0$-dimensional complex, and a connected graph with at least one edge, are top-concentrated.

\medskip
\begin{proposition}
	\label{prop : difference complex}
	Let $M$ be an oriented matching on the vertex set of $\Delta_n$, and let $\Gamma_{v,v'}\subseteq\FF_p^{\,\alpha}$ for $(v,v')\in M$ be sets such that for every pair $(v,v')$ in $M$ the differences $\{g-g' : g,g'\in\Gamma_{v,v'}\}$ generate the group $\FF_p^{\,\alpha}$.
	Set $\LL:=\LL_M(\Gamma_{v,v'} : (v,v')\in M)$.
	Then:
	\begin{compactenum}[\rm (1)]
		\item $\LL$ is a $G$-invariant simplicial complex,
		\item for every pair $(v,v')\in M$, every $c\in\Lambda$ and every non-empty union $S$ of blocks and fixed labels there are $c',c''\in S$ with $(c,c')\in R_{\Gamma_{v,v'}}$ and $(c'',c)\in R_{\Gamma_{v,v'}}$,
		\item the induced simplicial complex $\LL[U]$ is top-concentrated of dimension $n$ for every $G$-box $U$. In particular, if $n\geq N$, then $\LL$ satisfies the assumption of Theorem~\ref{th : fixed sphere}.
	\end{compactenum}
\end{proposition}

\begin{proof}
	Fix a pair $(v,v')\in M$ and write $R:=R_{\Gamma_{v,v'}}$ and $\Gamma:=\Gamma_{v,v'}$.
	Since the differences of $\Gamma$ generate the non-trivial group $\FF_p^{\,\alpha}$, the set $\Gamma$ is not empty.

	\medskip
	(1) A translation $g$ maps $\big((u,\iota),(u+h,\iota')\big)$ to $\big((u+g,\iota),(u+g+h,\iota')\big)$, which again lies in $R$, and it fixes the labels $\infty_\ell$.
	So $R$ is invariant under the diagonal action of $G$, and $\LL$ is $G$-invariant.

	\medskip
	(2) If $S$ contains a fixed label, take $c'=c''$ to be this label.
	Otherwise $S$ contains a block $\FF_p^{\,\alpha}\times\{\iota'\}$.
	If $c$ is a fixed label, take any $c'=c''\in S$.
	For $c=(u,\iota)$ take $h\in\Gamma$ and set $c'=(u+h,\iota')$ and $c''=(u-h,\iota')$.

	\medskip
	(3) Let $U$ be a $G$-box and $U_j:=U\cap Q_j$, viewed as a subset of $\Lambda$.
	Then $\LL[U]$ is the join of the bipartite graphs $R_{\Gamma_{v,v'}}[U_j,U_{j'}]$ for $(v,v')=(p_j,p_{j'})\in M$ and of the non-empty $0$-dimensional complexes $U_j$ for the vertices $p_j$ not covered by $M$.
	
	By the discussion above it suffices to show that the graph $R[S,S']$ is connected for all non-empty unions $S,S'$ of blocks and fixed labels.
	According to~(2) the graph $R[S,S']$ has at least an edge, and so $\LL[U]$ has dimension $n$.
	Additionally, property~(2) shows that every vertex of $R[S,S']$ has a neighbour in the other side.
	If $S'$ contains a fixed label $\infty_\ell$, then $\infty_\ell$ is adjacent to all vertices of $S$, and every vertex of $S'$ has a neighbour in $S$, so $R[S,S']$ is connected. The case that $S$ contains a fixed label is analogous.
	Otherwise $S=\FF_p^{\,\alpha}\times I$ and $S'=\FF_p^{\,\alpha}\times I'$ for non-empty sets $I,I'\subseteq[k]$.
	Recall that $R[S,S']$ is an undirected graph whose edges join a vertex $c\in S$ to a vertex $c'\in S'$ with $(c,c')\in R$.
	A path may traverse an edge in either direction, and we write $c\sim c'$ when $c$ and $c'$ are adjacent.
	For $h,h'\in\Gamma$, $\iota,\iota''\in I$ and $\iota'\in I'$ we have
	\[
		(u,\iota)\ \sim\ (u+h,\iota')\ \sim\ (u+h-h',\iota''),
	\]
	because $\big((u,\iota),(u+h,\iota')\big)\in R$ and $\big((u+h-h',\iota''),(u+h,\iota')\big)\in R$.
	This is a path of length two in $R[S,S']$ from $S$ to $S'$ and back to $S$; it changes the $\FF_p^{\,\alpha}$-label by the difference $h-h'$ of two elements of $\Gamma$, and it changes the block index within $I$ arbitrarily.
	Concatenating such paths, we see that the connected component of $(u,\iota)$ in $R[S,S']$ contains every vertex $(u+w,\iota'')\in S$ for which $w$ is a finite sum of differences of elements of $\Gamma$.
	These sums form the subgroup of $\FF_p^{\,\alpha}$ generated by the differences, which is $\FF_p^{\,\alpha}$ by assumption.
	Hence for every $u'\in\FF_p^{\,\alpha}$ and every $\iota''\in I$ the vertex $(u',\iota'')$ lies in the component of $(u,\iota)$, that is, all of $S$ lies in one connected component.
	Every vertex of $S'$ has a neighbour in $S$ by~(2), so this component contains $S'$ as well, and the graph $R[S,S']$ is connected.
\end{proof}

\medskip
The heart of the proof of Proposition~\ref{prop : difference complex}(3) is the connectivity of the graph $R_\Gamma[S,S']$ when $S$ and $S'$ are single blocks, illustrated in Figure~\ref{fig : difference graph} for $p=3$ and $\alpha=1$.

\medskip
The hypothesis on the differences is also necessary for~(3), for a suitable $G$-box $U$.
Indeed, suppose that for some pair $(v,v')=(p_j,p_{j'})\in M$ the differences of $\Gamma:=\Gamma_{v,v'}$ generate a proper subgroup $H$ of $\FF_p^{\,\alpha}$; recall that every block has $q\geq2$ elements.
Let $U$ be the $G$-box with $U_\ell:=U\cap Q_\ell=\FF_p^{\,\alpha}\times\{1\}$ for all $0\leq \ell\leq n$.
If $\Gamma=\emptyset$, the graph $R_\Gamma[U_j,U_{j'}]$ has no edges, and $\LL[U]$ has dimension smaller than $n$.
Otherwise $\Gamma\neq\emptyset$.
The graph $R_\Gamma[U_j,U_{j'}]$ has as its vertex set two copies of the block $\FF_p^{\,\alpha}\times\{1\}$, one in the column $Q_j$ and one in the column $Q_{j'}$, and a vertex $(u,1)$ of the first copy is adjacent to the vertices $(u+g,1)$, $g\in\Gamma$, of the second copy.
A path in this graph alternates between the two copies.
In particular, a path of length two from a vertex $(u,1)$ of the first copy back to the first copy passes through a vertex $(u+g,1)$ of the second copy, $g\in\Gamma$, and ends at a vertex $(u+g-g',1)$ with $g'\in\Gamma$.
Hence, it changes the $\FF_p^{\,\alpha}$-label by a difference of two elements of $\Gamma$, and a path of even length changes it by a sum of such differences, that is, by an element of $H$.
Consequently the connected component of the vertex $(0,1)$ of the first copy meets the first copy exactly in the vertices $(w,1)$ with $w\in H$.
Since $H\neq\FF_p^{\,\alpha}$, this component misses the vertices $(u,1)$ of the first copy with $u\notin H$, so the graph $R_\Gamma[U_j,U_{j'}]$ is disconnected and $\widetilde H_0(R_\Gamma[U_j,U_{j'}])\neq0$, although the graph has dimension $1$.
Every other join factor of $\LL[U]$ has non-trivial reduced homology as well: a $0$-dimensional factor $U_\ell$ consists of $q\geq2$ points, and a bipartite graph factor has $2q$ vertices and a number of edges divisible by $q$, hence different from $2q-1$, so it is not a tree.
By the K\"unneth formula~\eqref{eq : Kunneth for joins}, either $\LL[U]$ has dimension smaller than $n$, or it has non-trivial reduced homology in a degree smaller than $n$; in both cases it is not top-concentrated of dimension $n$.

\medskip
It is important to note that block indices are not constrained in the definition of $R_\Gamma$.
For example, if one also required $\iota=\iota'$, then for $k\geq2$ the graph between two different blocks would have no edges.
If the two columns of every pair of $M$ carry two different single blocks of a $G$-box $U$, then $\LL[U]$ would be a join of non-empty $0$-dimensional complexes, of dimension $n-|M|$ and with non-trivial reduced homology in this degree.
So the assumption of Theorem~\ref{th : fixed sphere} would fail whenever $n-|M|\leq N-2$, for instance for $n=N$ and the matchings with $\lfloor(N+1)/2\rfloor$ pairs used in Section~\ref{sec : counting}, since $k\geq2$ forces $N\geq6$.

%----------------------------------------------------------------------------------------------
\medskip
\subsection{Proof of Theorem~\ref{th : difference}}
%----------------------------------------------------------------------------------------------
We prove the theorem for affine maps $a\colon\Delta_n\longrightarrow\R^d$ with $n\geq N$.
Let $\LL:=\LL_M(\Gamma)\subseteq(\Delta_n)^{*r}_{\Delta(2)}$ be the difference complex of Definition~\ref{def : difference complex}.
By Proposition~\ref{prop : difference complex}, $\LL$ is $G$-invariant and satisfies the assumption of Theorem~\ref{th : fixed sphere}.
Theorem~\ref{th : fixed sphere} therefore gives a Tverberg face $J$ of $a$ in $\LL$.

\medskip
We extend $J$ to a facet $J'$ of $(\Delta_n)^{*r}_{\Delta(2)}$ which lies in $\LL$.

For a pair $(v,v')\in M$ such that $J$ labels exactly one of $v$ and $v'$, we label the other vertex with a label related to the given one by $R_\Gamma$, in the correct order; for a pair such that $J$ labels neither $v$ nor $v'$, we label them by $c$ and $c'$ with $(c,c')\in R_\Gamma$. This is possible by Proposition~\ref{prop : difference complex}(2) with $S=\Lambda$.
The vertices not covered by $M$ and not labelled by $J$ get arbitrary labels.
Then $J'\in\LL$, and $J'$ is a Tverberg face of $a$ by Lemma~\ref{lem : CSTM}, since $J'\supseteq J$.

\medskip
So $J'$ is the Tverberg face of a Tverberg partition $(F_c : c\in\Lambda)$ of $a$ whose faces contain all vertices of $\Delta_n$.
If $(v,v')\in M$, $v\in F_{(u,\iota)}$ and $v'\in F_{(u',\iota')}$, then $J'\in\LL$ gives $\big((u,\iota),(u',\iota')\big)\in R_\Gamma$, that is, $u'-u\in\Gamma$.
\qed

%----------------------------------------------------------------------------------------------
\medskip
\subsection{Proof of Corollary~\ref{cor : matching}}
%----------------------------------------------------------------------------------------------
We prove the corollary for affine maps $a\colon\Delta_n\longrightarrow\R^d$ with $n\geq N$.
If $r$ is not a power of $2$, an odd prime $p$ divides $r$.
Otherwise $r=2^\ell$ with $\ell\geq2$, and $4$ divides $r$.
So there is a prime power $q=p^\alpha\geq3$ which divides $r$.
Let $k=r/q$, $s=0$, and
\begin{compactitem}[\quad --]
	\item $\Gamma:=\{1,2\}\subseteq\FF_p$  if $\alpha=1$, and
	\item $ \Gamma:=\{e_1,\dots,e_\alpha,e_1+e_2\}\subseteq\FF_p^{\,\alpha}$ if $\alpha\geq2$.
\end{compactitem}
If $\alpha=1$, then $p\geq3$, so $0\notin\Gamma$, and $2-1=1$ generates $\FF_p$.
For $\alpha\geq2$ we have $0\notin\Gamma$, and the differences generate $\FF_p^{\,\alpha}$, since $(e_1+e_2)-e_2=e_1$ and $e_i=(e_i-e_1)+e_1$.

Orient the edges of the matching arbitrarily, and apply Theorem~\ref{th : difference}.
Since $s=0$, every face of the Tverberg partition is labelled by an element of $\FF_p^{\,\alpha}\times[k]$, and for every edge the two vertices lie in faces whose $\FF_p^{\,\alpha}$-labels differ by an element of $\Gamma$; as $0\notin\Gamma$, the two faces are different.
\qed

%----------------------------------------------------------------------------------------------
\medskip
\subsection{Counting with translated difference complexes}
\label{sec : counting}
%----------------------------------------------------------------------------------------------

For the proof of Theorem~\ref{th : main} we count the Tverberg faces of an affine map with the help of translated difference complexes. Our approach is as follows.
In Lemma~\ref{lem : free relabelling} we see that the action of the group $G$ permutes the Tverberg faces freely, then the translated difference complexes which contain a given facet are counted in Lemma~\ref{lem : number of translations}, and the exponent appearing in this count is then bounded by Lemma~\ref{lem : number of free pairs}.
Throughout this subsection let $n=N$.

\medskip
\begin{lemma}
	\label{lem : free relabelling}
	Let $a\colon\Delta_N\longrightarrow\R^d$ be an affine map.
	\begin{compactenum}[\rm (1)]
		\item Every unordered Tverberg partition of $a$ gives exactly $r!$ Tverberg faces of $a$.
		\item The group $G$ acts freely on the set of Tverberg faces of $a$ by relabelling.
		\item If $a$ is in general position, then every Tverberg face of $a$ is a facet of $\K$ in which every label is used at most $d+1$ times.
	\end{compactenum}
\end{lemma}
\begin{proof}
	(1) An unordered Tverberg partition $\{F_1,\dots,F_r\}$ gives one Tverberg face for each bijection between the faces $F_i$ and the labels, and these faces of $\K$ are pairwise different, because the faces $F_i$ are non-empty and pairwise disjoint.
	(2) Acting by $g\in G$ does not change the faces of the partition, so it maps Tverberg faces to Tverberg faces.
	If $g\neq0$, then $g$ gives the label $(g,1)\neq(0,1)$ to the face with the label $(0,1)$, and so it moves every Tverberg face.
	(3) This is the discussion of general position in Section~\ref{sec : introduction}: in general position a Tverberg partition uses all vertices of $\Delta_N$, and each of its faces has at most $d+1$ vertices.
\end{proof}

\medskip
Let $M$ be an oriented matching on the vertex set of $\Delta_N$, and $\Gamma\subseteq\FF_p^{\,\alpha}$.
For a family of translations $h=(h_{v,v'} : (v,v')\in M) \in(\FF_p^{\,\alpha})^M$ we consider the difference complex
\[
	\LL_h:=\LL_M(\Gamma+h_{v,v'} : (v,v')\in M),
\]
where $\Gamma+h_{v,v'}=\{g+h_{v,v'} : g\in\Gamma\}$.
There are $q^{|M|}$ families $h$.

\medskip
Recall that a facet $J$ of $\K$ is a labelling of all vertices of $\Delta_N$ by $\Lambda$, the vertex $p_j$ receiving the label $c$ when $p_{j,c}\in J$.
For a facet $J$ of $\K$ we denote by $X_M(J)$ the number of pairs $(v,v')\in M$ such that the labels of $v$ and of $v'$ in $J$ both lie in blocks, that is, both are different from $\infty_1,\dots,\infty_s$.
Clearly $X_M(J)=|M|$ when $J$ uses no fixed label.
Figure~\ref{fig : X} shows two facets with different values of $X_M$.

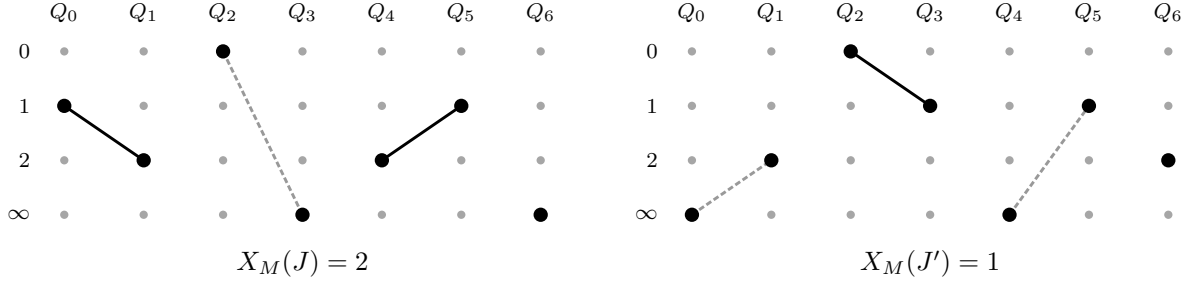
\begin{figure}[ht]
\centering
\begin{tikzpicture}[x=10.5mm,y=7.2mm,
	pt/.style={circle,fill=black!35,inner sep=1.1pt},
	sel/.style={circle,fill=black,inner sep=1.9pt},
	lab/.style={font=\scriptsize},
	cnt/.style={line width=1.1pt},
	ncnt/.style={line width=1.1pt,black!40,dash pattern=on 2.2pt off 1.1pt}]
	\begin{scope}
		\foreach \j in {0,...,6}{\foreach \y in {0,-1,-2,-3}{\node[pt] at (\j,\y) {};} \node[lab] at (\j,0.7) {$Q_{\j}$};}
		\foreach \y/\l in {0/0,-1/1,-2/2,-3/\infty}{\node[lab,anchor=east] at (-0.3,\y) {$\l$};}
		\draw[cnt] (0,-1) -- (1,-2);
		\draw[ncnt] (2,0) -- (3,-3);
		\draw[cnt] (4,-2) -- (5,-1);
		\foreach \j/\y in {0/-1,1/-2,2/0,3/-3,4/-2,5/-1,6/-3}{\node[sel] at (\j,\y) {};}
		\node[font=\small] at (3,-3.9) {$X_M(J)=2$};
	\end{scope}
	\begin{scope}[xshift=8.3cm]
		\foreach \j in {0,...,6}{\foreach \y in {0,-1,-2,-3}{\node[pt] at (\j,\y) {};} \node[lab] at (\j,0.7) {$Q_{\j}$};}
		\foreach \y/\l in {0/0,-1/1,-2/2,-3/\infty}{\node[lab,anchor=east] at (-0.3,\y) {$\l$};}
		\draw[ncnt] (0,-3) -- (1,-2);
		\draw[cnt] (2,0) -- (3,-1);
		\draw[ncnt] (4,-3) -- (5,-1);
		\foreach \j/\y in {0/-3,1/-2,2/0,3/-1,4/-3,5/-1,6/-2}{\node[sel] at (\j,\y) {};}
		\node[font=\small] at (3,-3.9) {$X_M(J')=1$};
	\end{scope}
\end{tikzpicture}
\caption{\small The number $X_M(J)$ for $r=4=1\cdot3+1$, $\Lambda=\FF_3\sqcup\{\infty\}$, $d=1$, so $N=6$, and the oriented matching $M=\{(p_0,p_1),(p_2,p_3),(p_4,p_5)\}$ with $\lfloor(N+1)/2\rfloor=3$ pairs. The vertex $p_6$ is not covered by $M$.
A facet of $\K$ chooses one vertex in each column, drawn as a large dot, and the pairs of $M$ are indicated by the segments joining the chosen vertices of the two columns of a pair.
A pair contributes to $X_M$ when both chosen labels lie in the block $\FF_3$ (solid segments), and it does not when one of them is $\infty$ (dashed segments).
The facet $J$ on the left labels $p_0,\dots,p_6$ by $1,2,0,\infty,2,1,\infty$ and has $X_M(J)=2$, whereas the facet $J'$ on the right labels them by $\infty,2,0,1,\infty,1,2$ and has $X_M(J')=1$. That is the lower bound $(r-1-2s)(d+1)/2=1$ of Lemma~\ref{lem : number of free pairs}, since $J'$ uses the label $\infty$ the maximal number $d+1=2$ of times, both times on matched vertices.
For $\Gamma=\{0,1\}$ the number of translation families $h$ with $J\in\LL_h$ is $|\Gamma|^{X_M(J)}q^{|M|-X_M(J)}=2^2\cdot3=12$ by Lemma~\ref{lem : number of translations}, and for $J'$ it is $2\cdot3^2=18$.}
\label{fig : X}
\end{figure}

\medskip
\begin{lemma}
	\label{lem : number of translations}
	For every facet $J$ of $\K$ the number of families $h\in(\FF_p^{\,\alpha})^M$ with $J\in\LL_h$ is
	\[
		|\Gamma|^{X_M(J)}\cdot q^{|M|-X_M(J)} .
	\]
\end{lemma}
\begin{proof}
	By Definition~\ref{def : difference complex}, the facet $J$ lies in $\LL_h$ if and only if for every pair $(v,v')\in M$ the labels $c$ of $v$ and $c'$ of $v'$ in $J$ satisfy $(c,c')\in R_{\Gamma+h_{v,v'}}$.
	Thus, this is one condition for each pair, and the condition for the pair $(v,v')$ depends on the family $h$ only through its entry $h_{v,v'}$.
	For each pair $(v,v')\in M$, with the labels $c$ of $v$ and $c'$ of $v'$ in $J$, we define
	\[
		\Omega_{v,v'}:=\big\{t\in\FF_p^{\,\alpha} : (c,c')\in R_{\Gamma+t}\big\},
	\]
	the set of those translations $t$ of $\Gamma$ for which the pair $(v,v')$, labelled as in $J$, satisfies the constraint of the translated set $\Gamma+t$.
	The set $\Omega_{v,v'}$ depends on $J$ through the two labels $c$ and $c'$ only.
	With this notation, the condition for the pair $(v,v')$ in the family $h$ holds exactly when $h_{v,v'}\in \Omega_{v,v'}$, so
	\[
		J\in\LL_h\quad\Longleftrightarrow\quad h_{v,v'}\in \Omega_{v,v'}\ \text{ for every }(v,v')\in M .
	\]
	In other words, the families $h=(h_{v,v'} : (v,v')\in M)$ with $J\in\LL_h$ are exactly the elements of the product $\prod_{(v,v')\in M}\Omega_{v,v'}\subseteq(\FF_p^{\,\alpha})^M$, whose number is $\prod_{(v,v')\in M}|\Omega_{v,v'}|$.

	\smallskip
	It remains to determine $|\Omega_{v,v'}|$.
	If $c$ or $c'$ is a fixed label, then $(c,c')\in R_{\Gamma'}$ for every $\Gamma'$, so the condition holds for every value of $h_{v,v'}$ and so $|\Omega_{v,v'}|=q$.
	If $c=(u,\iota)$ and $c'=(u',\iota')$ both lie in blocks, then $(c,c')\in R_{\Gamma+h_{v,v'}}$ means $u'-u\in\Gamma+h_{v,v'}$, that is, $u'-u=g+h_{v,v'}$ for some $g\in\Gamma$, or equivalently $h_{v,v'}\in(u'-u)-\Gamma$.
	The set $(u'-u)-\Gamma=\{(u'-u)-g : g\in\Gamma\}$ has exactly $|\Gamma|$ elements, and so $|\Omega_{v,v'}|=|\Gamma|$.
	The second case occurs for exactly $X_M(J)$ pairs and the first case for the remaining $|M|-X_M(J)$ pairs, which gives the product $|\Gamma|^{X_M(J)}\cdot q^{|M|-X_M(J)}$.
\end{proof}

\medskip
\begin{lemma}
	\label{lem : number of free pairs}
	Let $r=kq+s$ be the fixed decomposition of this section, so that $\Lambda$ has the $s$ fixed labels $\infty_1,\dots,\infty_s$.
	Let $M$ be an oriented matching on the vertex set of the simplex $\Delta_N$ with $\lfloor\frac{N+1}2\rfloor$ pairs, and let $J$ be a facet of $\K$ in which each of the labels $\infty_1,\dots,\infty_s$ may appear at most $d+1$ times.
	Then
	\[
		X_M(J)\ \geq\ \frac{(r-1-2s)(d+1)}{2}.
	\]
\end{lemma}
\begin{proof}
	Let $s_0$ be the number of vertices of $\Delta_N$ whose label in $J$ is one of $\infty_1,\dots,\infty_s$.
	Since each of these $s$ labels is used at most $d+1$ times, $s_0\leq s(d+1)$, and the other $N+1-s_0$ vertices have labels in blocks.
	
	Let $B$ be the set of these $N+1-s_0$ vertices with labels in blocks.
	We sort the vertices of $B$ according to their position with respect to the matching $M$.
	Since $M$ has $\lfloor\frac{N+1}2\rfloor$ pairs, at most one vertex of $\Delta_N$, and hence at most one vertex of $B$, is not covered by $M$.
	A vertex of $B$ covered by $M$ lies in a pair $(v,v')$ whose other vertex either has a fixed label or lies in $B$ as well. We say that the pair is of \textit{the first kind} in the former case and of \textit{the second kind} in the latter.
	
	Every pair of the first kind contains a vertex with a fixed label, and it contains exactly one vertex of $B$, namely its other vertex.
	Since $M$ is a matching, every vertex of $\Delta_N$ belongs to at most one pair of $M$.

	Thus the number of pairs of the first kind is at most the number $s_0$ of such vertices.
	
	The pairs of the second kind are those in which both labels lie in blocks, so by definition there are $X_M(J)$ of them, and together they contain $2X_M(J)$ vertices of $B$.
	Counting the vertices of $B$ in these three groups gives
	\[
		N+1-s_0\ =\ |B|\ \leq\ 1+s_0+2X_M(J),
	\]
	and therefore
	\[
		2X_M(J)\ \geq\ N-2s_0\ \geq\ (d+1)(r-1)-2s(d+1)\ =\ (r-1-2s)(d+1),
	\]
	where the second inequality is $s_0\leq s(d+1)$.
\end{proof}

%----------------------------------------------------------------------------------------------
\medskip
\subsection{Proof of Theorem~\ref{th : main}}
%----------------------------------------------------------------------------------------------

The proof is a double counting of pairs composed of a translated difference complex and a Tverberg face of $a$ it contains.
Since the claimed lower bound does not depend on $a$, and since the affine maps in general position are dense in the space of all affine maps $\Delta_N\longrightarrow\R^d$, by the discussion of general position in Section~\ref{sec : introduction} it suffices to prove the bound for maps in general position, and so we may assume that $a$ is in general position.
Let $\Gamma:=\{0,e_1,\dots,e_\alpha\}\subseteq\FF_p^{\,\alpha}$, and let $M$ be an oriented matching on the vertex set of $\Delta_N$ with $\lfloor(N+1)/2\rfloor$ pairs.
We consider the $q^{|M|}$ complexes $\LL_h$ where $h\in(\FF_p^{\,\alpha})^M$, as in Section~\ref{sec : counting}.
The differences of every set $\Gamma+h_{v,v'}$ are the differences of $\Gamma$, and they generate $\FF_p^{\,\alpha}$, since $e_i-0=e_i$.
By Proposition~\ref{prop : difference complex} every $\LL_h$ is $G$-invariant and satisfies the assumption of Theorem~\ref{th : fixed sphere}, and so it contains a Tverberg face of $a$ by Theorem~\ref{th : fixed sphere}.

\medskip
Denote by $\mathcal T$ the set of Tverberg faces of $a$.
According to  Lemma~\ref{lem : free relabelling}(1) we have that $|\mathcal T|=r!\cdot\Tv(a)$.
We count the set $P$ of pairs $(h,J)$ with $h\in(\FF_p^{\,\alpha})^M$, $J\in\mathcal T$ and $J\in\LL_h$ in two ways.

\medskip
On the one hand, every $\LL_h$ contains a Tverberg face of $a$, and with it the whole $G$-orbit of this face, because $\LL_h$ is $G$-invariant.
This orbit has $q$ elements by Lemma~\ref{lem : free relabelling}(2).
Hence $|P|\geq q\cdot q^{|M|}$.

\medskip
On the other hand, let $J\in\mathcal T$.
By Lemma~\ref{lem : free relabelling}(3), $J$ is a facet of $\K$ in which every label is used at most $d+1$ times, so $X_M(J)\geq\lceil\frac{(r-1-2s)(d+1)}2\rceil$ by Lemma~\ref{lem : number of free pairs}, as $X_M(J)$ is an integer.
Since $|\Gamma|=\alpha+1\leq q$, Lemma~\ref{lem : number of translations} shows that $J$ lies in exactly
\[
	(\alpha+1)^{X_M(J)}\,q^{|M|-X_M(J)}\ =\ q^{|M|}\Big(\frac{\alpha+1}{q}\Big)^{X_M(J)}\ \leq\ q^{|M|}\Big(\frac{\alpha+1}{q}\Big)^{\lceil\frac{(r-1-2s)(d+1)}2\rceil}
\]
of the complexes $\LL_h$.
The set $P$ is the disjoint union, over $J\in\mathcal T$, of the sets $\{(h,J) : J\in\LL_h\}$, whose sizes were just bounded, so summing over the $|\mathcal T|=r!\cdot\Tv(a)$ Tverberg faces gives
\[
	|P|\ =\ \sum_{J\in\mathcal T}\big|\{h\in(\FF_p^{\,\alpha})^M : J\in\LL_h\}\big|\ \leq\ r!\cdot\Tv(a)\cdot q^{|M|}\Big(\frac{\alpha+1}{q}\Big)^{\lceil \frac{(r-1-2s)(d+1)}2\rceil}.
\]

\medskip
Comparing the two estimates for $|P|$ we get
\[
	\Tv(a)\ \geq\ \frac{q}{r!}\,\Big(\frac{q}{\alpha+1}\Big)^{\lceil\frac{(r-1-2s)(d+1)}2\rceil}.
\]
\qed

%----------------------------------------------------------------------------------------------
\medskip
\subsection{Proof of Corollary~\ref{cor : every r}}
%----------------------------------------------------------------------------------------------
Apply Theorem~\ref{th : main} with $q=b(r)$, $k=\frac{r}{b(r)}$ and $s=0$.
Then $(r-1-2s)(d+1)=N$.
\qed

%----------------------------------------------------------------------------------------------
\medskip
\subsection{Proof of Corollary~\ref{cor : rate}}
%----------------------------------------------------------------------------------------------
Let $r=kq_r+s$ with $q_r=p^\alpha$ be a decomposition at which the maximum $\beta_r$ is attained, and let $x:=\frac{(r-1-2s)(d+1)}2$.
Since $\frac{q_r}{\alpha+1}\geq1$, we have that 
\[
\Big(\frac{q_r}{\alpha+1}\Big)^{\lceil x\rceil}\geq\Big(\frac{q_r}{\alpha+1}\Big)^{x}=\beta_r^{\,d+1}.
\]
Thus, from Theorem~\ref{th : main} it follows that
\[
\Tv(a)\geq\frac{q_r}{r!}\cdot \beta_r^{\,d+1},
\] 
for every affine map $a\colon\Delta_N\longrightarrow\R^d$.
Consequently, $\Tv_r(d)^{\frac1{d}}\geq \big(\frac{q_r}{r!}\big)^{\frac1{d}}\,\beta_r^{\frac{d+1}{d}}$, and the right side tends to $\beta_r$ as $d\to\infty$.
This gives the lower bound for the $\liminf$.

\medskip
The upper bound for the $\limsup$ follows from $\Tv_r(d)\leq((r-1)!)^d$, which holds because the $(d+1)$-clouds configuration of Section~\ref{sec : introduction} is an affine map $\Delta_N\longrightarrow\R^d$ with exactly $((r-1)!)^d$ unordered Tverberg partitions.

\medskip
Finally, $\beta_r>1$.
Let $b(r)=p^\alpha$.
As in the proof of Corollary~\ref{cor : matching}, $b(r)\geq3$.
If $p\geq3$, then $p^\alpha\geq3^\alpha>\alpha+1$.
On the other hand, if $p=2$, then $\alpha\geq2$ and $2^\alpha>\alpha+1$.
So the decomposition $r=\frac{r}{b(r)}\cdot b(r)+0$ gives $\beta_r\geq 
\big(\frac{b(r)}{\alpha+1}\big)^{(r-1)/2}>1$.
\qed

%----------------------------------------------------------------------------------------------
\medskip
\subsection{Proof of Theorem~\ref{th : averaged}}
\label{subsec : proof averaged}
%----------------------------------------------------------------------------------------------
As in the proof of Theorem~\ref{th : main} we may assume that $a$ is in general position, and we use $\Gamma=\{0,e_1,\dots,e_\alpha\}$, the set $\mathcal T$ of Tverberg faces of $a$, and the complexes $\LL_h$, now for every oriented matching $M$ in the set $\mathcal M$ of all oriented matchings on the vertex set of $\Delta_N$ with $\lfloor\frac{N+1}2\rfloor$ pairs.
For $M\in\mathcal M$ let $P_M$ be the set $P$ of the proof of Theorem~\ref{th : main}, instead constructed from the oriented matching $M$.

\medskip
Every permutation $\sigma$ of $[k]$ and every permutation $\tau$ of $\{\infty_1,\dots,\infty_s\}$ give the permutation of $\Lambda$ which maps $(u,\iota)$ to $(u,\sigma(\iota))$ and $\infty_\ell$ to $\tau(\infty_\ell)$, and it commutes with the action of $G$.
Since the relations $R_{\Gamma'}$ do not involve block indices and treat all fixed labels the same, every complex $\LL_h$ is invariant under the group of order $q\,k!\,s!$ generated by $G$ and these permutations.
This group acts freely on $\mathcal T$, because a relabelling which fixes a Tverberg face fixes each of its faces, which are non-empty and pairwise disjoint, and therefore fixes every label.
Hence every $\LL_h$ contains at least $q\,k!\,s!$ Tverberg faces, and $|P_M|\geq q\,k!\,s!\,q^{|M|}$ for every $M\in\mathcal M$.
On the other hand, by Lemma~\ref{lem : number of translations} a Tverberg face $J$ lies in exactly $q^{|M|}\big(\frac{\alpha+1}q\big)^{X_M(J)}$ of the complexes $\LL_h$ built from $M$.
Summing over $\mathcal M$ we obtain
\begin{equation}\label{eq : averaged count}
	|\mathcal M|\,q\,k!\,s!\ \leq\ \sum_{J\in\mathcal T}\ \sum_{M\in\mathcal M}\Big(\frac{\alpha+1}q\Big)^{X_M(J)}.
\end{equation}

\medskip
For $J\in\mathcal T$ let $F_J$ be the set of vertices of $\Delta_N$ whose labels in $J$ are fixed labels.
Then $X_M(J)$ is the number of pairs of $M$ disjoint from $F_J$. It does not increase when $F_J$ is enlarged, and the uniform measure on $\mathcal M$ is invariant under permutations of the vertices of $\Delta_N$.
So the inner sum in~\eqref{eq : averaged count} depends only on $|F_J|$, and it does not decrease as $|F_J|$ grows.
By Lemma~\ref{lem : free relabelling}(3) we have $|F_J|\leq s(d+1)\leq N+1$, where the second inequality follows from $s\leq r-2$.
Let $F$ be a set of $s(d+1)$ vertices, let $X_M$ be the number of pairs of $M$ disjoint from $F$, and set $b:=N+1-s(d+1)$.
Then~\eqref{eq : averaged count} and $|\mathcal T|=r!\cdot\Tv(a)$ give
\[
	\Tv(a)\ \geq\ \frac{q\,k!\,s!}{r!}\,\Big(\frac{1}{|\mathcal M|}\sum_{M\in\mathcal M}\Big(\frac{\alpha+1}q\Big)^{X_M}\Big)^{-1}.
\]

\medskip
It remains to compute this average.
Every matching on the vertex set of $\Delta_N$ with $|M|$ pairs has exactly $2^{|M|}$ orientations, and $X_M$ depends only on the underlying matching, so the average over $\mathcal M$ equals the average over the set $\mathcal M'$ of all matchings with $|M|$ pairs, without orientation.
For integers $x,w,y\geq 0$ let $\mathcal{M}'_{x,w,y}\subseteq\mathcal{M}'$ be the set of those matchings which have exactly $x$ pairs disjoint from $F$, exactly $w$ pairs contained in $F$ and $y$ pairs meeting both $F$ and its complement.
These sets partition $\mathcal M'$, and $X_M=x$ for every $M\in\mathcal{M}'_{x,w,y}$, so
\begin{equation}
	\label{eq : average as sum}
	\frac{1}{|\mathcal M|}\sum_{M\in\mathcal{M}}\Big(\frac{\alpha+1}q\Big)^{X_M}\ =\ \sum_{x,w,y\geq0}\frac{|\mathcal M'_{x,w,y}|}{|\mathcal M'|}\Big(\frac{\alpha+1}q\Big)^{x}.
\end{equation}
For $M\in\mathcal M'_{x,w,y}$ the numbers $b-2x-y$ and $N+1-b-2w-y$ of unmatched vertices outside $F$ and inside $F$ respectively sum up to $N+1-2|M|\in\{0,1\}$, since $M$ is assumed to have $\lfloor\frac{N+1}2\rfloor$ pairs. Hence $\mathcal M'_{x,w,y}\neq\emptyset$ implies $x+w+y=|M|$ and $2x+y\in\{b,\,b-(N+1-2|M|)\}$, which is either $\{b\}$ or $\{b,b-1\}$.

Conversely, let $x,w,y\geq0$ be integers with $x+w+y=|M|$ and $2x+y\in\{b,\,b-(N+1-2|M|)\}$.
Then $b-2x-y\in\{0,\,N+1-2|M|\}$, and $N+1-b-2w-y=(N+1-2|M|)-(b-2x-y)\in\{N+1-2|M|,\,0\}$, so each of the two integers $b-2x-y$ and $N+1-b-2w-y$ is $0$ or $1$.
For such $x,w,y$ the set $\mathcal M'_{x,w,y}$ is not empty, and a matching in $\mathcal M'_{x,w,y}$ is obtained by
\begin{compactitem}[\quad --]
  \item choosing, among the $b$ vertices outside $F$, the $2x$ vertices of the pairs disjoint from $F$ together with their partition into $x$ pairs, the $y$ vertices of the pairs meeting both $F$ and its complement, and, if $b-2x-y=1$, the unmatched vertex outside $F$,
  \item choosing, among the $N+1-b$ vertices of $F$, the $2w$ vertices of the pairs contained in $F$ together with their partition into $w$ pairs, the $y$ vertices of the pairs meeting both sides, and, if $N+1-b-2w-y=1$, the unmatched vertex inside $F$,
  \item matching the $y$ chosen vertices outside $F$ with the $y$ chosen vertices of $F$ in one of $y!$ ways.
\end{compactitem}
Hence
\[
	|\mathcal M'_{x,w,y}|\ =\ \frac{b!}{2^{x}\,x!\,y!}\cdot\frac{(N+1-b)!}{2^{w}\,w!\,y!}\cdot y!\ =\ \frac{b!\,(N+1-b)!}{2^{x+w}\,x!\,w!\,y!},
\]
and together with $|\mathcal M'|=\frac{(N+1)!}{2^{|M|}\,|M|!}$ this gives
\[
	\frac{|\mathcal M'_{x,w,y}|}{|\mathcal M'|}\ =\ \binom{N+1}{b}^{-1}\frac{|M|!}{x!\,w!\,y!}\,2^{y}.
\]
On the other hand, by the multinomial theorem,
\[
	\Big(1+2u+\frac{\alpha+1}q\,u^2\Big)^{|M|}\ =\ \sum_{x+w+y=|M|}\frac{|M|!}{x!\,w!\,y!}\,2^{y}\Big(\frac{\alpha+1}q\Big)^{x}u^{2x+y},
\]
and $(1+u)^{N+1-2|M|}=1+(N+1-2|M|)\,u$.
Consequently, the coefficient of $u^b$ in
\[
  (1+u)^{N+1-2|M|}\big(1+2u+\frac{\alpha+1}qu^2\big)^{|M|},
\]
is
\[
  \sum_{\substack{x+w+y=|M|\\ 2x+y\in\{b,\,b-(N+1-2|M|)\}}}\frac{|M|!}{x!\,w!\,y!}\,2^{y}\Big(\frac{\alpha+1}q\Big)^{x},
\]
where the sum runs over all integers $x,w,y\geq0$ satisfying the two conditions under the summation sign, that is, as shown above, over exactly those triples $(x,w,y)$ with $\mathcal M'_{x,w,y}\neq\emptyset$.
Thus the right side of~\eqref{eq : average as sum} is $\binom{N+1}{b}^{-1}$ times this coefficient, and since $\binom{N+1}{b}=\binom{N+1}{s(d+1)}$, this proves Theorem~\ref{th : averaged}.

\medskip
If $F$ is the set of vertices with fixed labels of a facet in which each fixed label is used exactly $d+1$ times, then Lemmas~\ref{lem : number of translations} and~\ref{lem : number of free pairs} give $X_M\geq\lceil\frac{(r-1-2s)(d+1)}2\rceil$ for every $M$.
So the average above is at most $\big(\frac{\alpha+1}q\big)^{\lceil\frac{(r-1-2s)(d+1)}2\rceil}$, and the bound of Theorem~\ref{th : averaged} is at least the bound of Theorem~\ref{th : main}.

%----------------------------------------------------------------------------------------------
\medskip
\subsection{Proof of Corollary~\ref{cor : averaged rate}}
%----------------------------------------------------------------------------------------------
By Corollary~\ref{cor : rate} it suffices to show $\liminf_{d\to\infty}\Tv_r(d)^{\frac{1}{d}}\geq\beta'_r$.
Fix a decomposition $r=kq+s$ with $s\geq1$ and fix $u>0$.
The polynomial in Theorem~\ref{th : averaged} has non-negative coefficients, so its coefficient of $u^{N+1-s(d+1)}$ is at most its value at $u$ divided by $u^{N+1-s(d+1)}$, and Theorem~\ref{th : averaged} gives
\[
	\Tv_r(d)\ \geq\ \frac{q\,k!\,s!}{r!}\,\binom{N+1}{s(d+1)}\,\frac{u^{N+1-s(d+1)}}{(1+u)^{N+1-2\lfloor\frac{N+1}2\rfloor}\big(1+2u+\frac{\alpha+1}q\,u^2\big)^{\lfloor\frac{N+1}2\rfloor}} .
\]
Since $N+1=(r-1)(d+1)+1$, Stirling's formula gives 
\[
\binom{N+1}{s(d+1)}^{\frac{1}{d}} \ \to \ (r-1)^{r-1}s^{-s}(r-1-s)^{-(r-1-s)}
\] 
as $d\to\infty$, and the $d$-th roots of the other factors converge to $1$, $u^{r-1-s}$, $1$ and $\big(1+2u+\frac{\alpha+1}q\,u^2\big)^{-\frac{r-1}2}$ respectively.
So $\liminf_{d\to\infty}\Tv_r(d)^{\frac{1}{d}}$ is at least the expression under the supremum in the definition of $\beta'_r$, for every $u>0$ and every decomposition with $s\geq1$, and therefore at least $\beta'_r$.

\medskip
The expression under the supremum tends to $0$ as $u\to0$ and as $u\to\infty$, because $1\leq s\leq r-2$.
Its logarithmic derivative is $\frac{r-1-s}u-\frac{(r-1)(1+\frac{\alpha+1}qu)}{1+2u+\frac{\alpha+1}qu^2}$, which vanishes if and only if $\frac{\alpha+1}{q}\,s\,u^2-(r-1-2s)\,u-(r-1-s)=0$.
This equation has exactly one positive root, since the product of its roots is negative, and the supremum is attained there.
\qed

 \providecommand{\href}[2]{#2}

\end{document}